\documentclass[11pt,reqno]{amsart}
\numberwithin{equation}{section}
\usepackage{amsmath,amsthm}
\usepackage{mathptmx}  
\usepackage{amssymb}
\usepackage{multicol}
\usepackage{bbm}
\DeclareFontFamily{U}{BOONDOX-calo}{\skewchar\font=45 }
\DeclareFontShape{U}{BOONDOX-calo}{m}{n}{
  <-> s*[1.05] BOONDOX-r-calo}{}
\DeclareFontShape{U}{BOONDOX-calo}{b}{n}{
  <-> s*[1.05] BOONDOX-b-calo}{}
\DeclareMathAlphabet{\mathcalboondox}{U}{BOONDOX-calo}{m}{n}
\SetMathAlphabet{\mathcalboondox}{bold}{U}{BOONDOX-calo}{b}{n}
\DeclareMathAlphabet{\mathbcalboondox}{U}{BOONDOX-calo}{b}{n}

\usepackage{pstricks}
\usepackage{float}
\usepackage{mathrsfs}
\usepackage[charter]{mathdesign}
\usepackage{tikz}
\usepackage{bbm}
\newcommand{\mcb}[1]{{\mathcalboondox #1}}
\usepackage{amssymb}
\usepackage{amsmath}
\usepackage{amsthm}
\usepackage{diagbox}
\usepackage{booktabs}
\usepackage{enumitem}
\usepackage[sans]{dsfont} % allows \mathds{E 1}, without sans is less heavy 
\usepackage{bbm}
\usepackage{changebar}
\usepackage[colorlinks]{hyperref}
\usepackage{a4wide}

\usepackage{tikz}
\usetikzlibrary{arrows,shapes,automata,petri,positioning,calc}
\tikzset{
    place/.style={
        circle,
        thick,
        draw=black,
        fill=gray!50,
        minimum size=20mm,
    },
        state/.style={
        circle,
        thick,
        draw=blue!75,
        fill=blue!20,
        minimum size=20mm,
    },
}

\tikzset{
    cross/.pic = {
    \draw[rotate = 45] (-0.2,0) -- (0.2,0);
    \draw[rotate = 45] (0,-0.2) -- (0, 0.2);
    }
}

\usepackage{ulem}
\usepackage{setspace}
\newtheorem{thm}{Theorem}[section]
\newtheorem{lem}[thm]{Lemma}
\newtheorem{cor}[thm]{Corollary}

\newtheorem{prop}[thm]{Proposition}

\newtheorem{definition}[thm]{Definition}

\newtheorem{rem}[thm]{Remark}

\newcommand\ve{\varepsilon}

\usepackage{comment}

\title[Fractional porous medium from  a microscopic dynamics]{Derivation of the  fractional porous medium equation\\ from  a microscopic dynamics}

\author{Pedro Cardoso, Renato de Paula, Patr\'icia   Gon\c calves}

\begin{document}

\begin{abstract}
In this article we derive the fractional porous medium equation for any power of the fractional Laplacian as the  hydrodynamic limit of a microscopic dynamics of random particles with long range interactions, but the jump rate highly depends on the occupancy near the sites where the interactions take place. 
\end{abstract}

\maketitle

\section{Introduction}
The rigorous mathematical derivation of the macroscopic evolution equations of classical fluid mechanics from the large-scale description of the conserved quantities in Newtonian particle systems is a long-standing problem in mathematical physics. Instead, if the deterministic dynamics is replaced by  a stochastic dynamics,  one can provide positive answers in this direction. Over the last three decades, there has been remarkable progress in deriving the well-known hydrodynamic limit, from stochastic interacting particle systems, by means of rigorous mathematical results.  In this framework, many partial differential equations (PDEs) have been studied and derived from several underlying random dynamics. The nature of these equations  highly depends on the chosen microscopic stochastic dynamics: it can be parabolic, hyperbolic, or even of fractional form. Our focus on this article is on the latter type of equations.

Inspired by the works of \cite{DQRV1,DQRV2} we focus on the fractional porous medium equation given for $\gamma\in(0,2)$ and $m\in\mathbb N$ by
\begin{equation}\label{eq:porosos}
\begin{cases}
&\partial_{t}\rho(t,u) =  [-(-\Delta)^{\gamma/2} \rho^m](t,u) , \;\; u\in \mathbb{R}, \; t \in [0,T], \\
&\rho(0,u) = g(u), \;\; u\in \mathbb{R},
\end{cases}
\end{equation}
where $g:\mathbb R \rightarrow [0,1]$ is a measurable function.
Equations of the form above can be seen as nonlinear versions of the linear fractional  heat equation obtained for the choice $ m =1$. For simplicity of the presentation in this article we restrict to the case $m=2$, nevertheless all our results trivially extend to the case $2\leq m\in\mathbb N$, the difference is that one needs to require more particles in the vicinity of the sites where the jumps occur. The porous medium equations model the  so-called anomalous diffusions and they have been extensively studied in the PDE's literature. The fractional Laplacian operator above can be defined as in \eqref{eq:frac_lap}  and it is the   infinitesimal generator of stable L\'evy processes, and contrarily to the usual Laplacian operator, it is  a non-local operator. Our interest is to obtain the fractional porous medium equation as the hydrodynamic limit of an interacting particle system.  The fractional heat equation  (corresponding to the choice $m=1$) was obtained  from an exclusion process with symmetric rates  and with long jumps in \cite{jara2009hydrodynamic}. There it is  also derived the general fractional equation $\partial_{t}\rho(t,u) =  [-(-\Delta)^{\gamma /2} \Phi(\rho)](t,u)$, where the function $\Phi$ satisfies some conditions (see Section 8.1 in that article) that do not cover the polynomial case that we treat here.

We observe that by replacing the fractional Laplacian by the usual Laplacian operator in the equation above, we  obtain the porous medium equation  which has been studied in \cite{vazquez_book}.  This last equation (for $m\geq 2$) has been derived as the hydrodynamic limit of an interacting particle system  of exclusion type, first in \cite{GLT} for the equation on the torus and later in \cite{MR4099999} for the  equation on the interval $[0,1]$ and with several boundary conditions of Dirichlet, Robin and Neumann type.

The underlying dynamics that was considered in \cite{GLT} in order to derive the porous medium equation on the one-dimensional torus $\mathbb T$, for $m=2,$ is described as follows. First one discretizes the torus $\mathbb T$ by a scaling parameter $n$ which will be taken to infinity. The discrete space where particles will evolve is the one dimensional discrete torus $\mathbb T_n:=\{0,1,2,\cdots, n-1\}$. At each site of $\mathbb T_n$ it is allowed at most one particle (the exclusion rule) and we denote the number of particles at site $x$ at any time $t$ by $\eta_t(x)$.  After an exponential clock of rate one,  particles in a bond $\{x,x+1\}$ exchange their positions with a rate $1$ if there is only a particle at $x-1$ or $x+2$ and with rate $2$ if there are particles in both sites.  To derive the porous medium equation with $m>2$ one has just to require at least $m$ particles in a vicinity of the exchanging particles.

The reason for the choice of  the rates given above is to work with  an underlying microscopic dynamics of gradient type. By gradient we mean that the instantaneous current of the system at any bond $\{x,x+1\}$, i.e. the difference between the jump rate from $x$ to $x+1$ and the jump rate from $x+1$ to $x$, that we denote by $j_{x,x+1}(\eta)$, is written as the gradient of some local function. More precisely, the jump rate from $x$ to $x+1$ is given by  $\xi_{x,x+1}(\eta)r_{x,x+1}(\eta)$ where $\xi_{x,x+1}(\eta):=\eta(x)[1-\eta(x+1)]$ is  the rate corresponding to the exclusion dynamics, which means it is non null if, and only if, only one of the points in the bond $\{x,x+1\}$  is occupied; and $r_{x,x+1}(\eta)=\eta(x-1)+\eta(x+2)$, which means that if no particles are present at the sites $x-1$ and $x+2$ the jump rate is null. From these definitions we see that $$j_{x,x+1}(\eta)= [\eta(x)-\eta(x+1)][\eta(x-1)+\eta(x+2)]=\tau_xh(\eta)-\tau_{x+1}h(\eta),$$ where $h(\eta)=\eta(x-1)\eta(x)+\eta(x)\eta(x+1)-\eta(x-1)\eta(x+1)$. Since the model is of gradient type when computing the discrete profile defined by $\rho_t^n(x)=\mathbb E	[\eta_t(x)]$, we have from Kolmogorov's equation that $  \partial_t\rho_t^n(x)=\mathbb E[\mcb L \eta_t(x)]$ where $\mcb  L$ is the infinitesimal generator of the Markov process $\eta_t$. From the conservation law it holds $\mcb  L\eta(x)=j_{x-1,x}(\eta)-j_{x,x+1}(\eta)$ and since $j_{x,x+1}$ itself is another gradient, we get  $  \partial_t \rho_t^n(x)=\Delta_n\mathbb E[\tau_x h(\eta)]$, where $~\Delta_n$ denotes the discrete Laplacian.  We note that for this model, the invariant state  is the Bernoulli product measure with a constant parameter. Above the expectation $\mathbb E$ is with respect to the  Bernoulli product measure, but with a parameter given by $\rho_t^n(\cdot)$. Note that the expectation of $h$ with respect to this measure is given by $$\mathbb E[\tau_xh(\eta)]=\rho_t^n(x-1)\rho_t^n(x)+\rho_t^n(x)\rho_t^n(x+1)-\rho_t^n(x-1)\rho_t^n(x+1).$$ Now, if we assume that  for all $x$ it holds  $\lim_{n  \rightarrow  \infty}\rho_t^n(x)=\rho_t(\tfrac xn)$, then we obtain that  the evolution of the density is given by the porous medium equation: $$  \partial_t \rho_t(u)=\Delta\rho_t^2(u).$$ 
The model introduced above belongs to the class of kinetically constrained lattice gases which consist of stochastic interacting particle systems with exclusion constraints whose exchange rates depend locally on the configuration. These models  have been introduced and analysed in the physical literature since the late 1980's and they model glassy dynamics, i.e. the liquid/glass transition. On the other hand, the  porous medium equation $  \partial_t \rho_t(u)=\Delta\rho_t^m(u)$  appears in different contexts in the physical literature since it models  the density of an ideal gas flowing isothermally through an homogeneous porous medium. Solutions of this equation  can be compactly supported at each fixed time  (finite speed of propagation) which is not the case of solutions to the heat equation. This is a consequence of the fact that the diffusion coefficient $D(\rho)=m\rho^{m-1}$ vanishes as $\rho \rightarrow 0$.

In this article with the aim of studying the fractional porous medium equation,  we propose a new model, which is an extension of the one just described, but in this case particles can give long  jumps in $\mathbb Z$ and with a rate that decreases as the jump size increases, in a similar fashion to the case $m=2$. More precisely, a particle jumps from a position $x$ to $x+z$ according to a  probability transition function  defined  on  $\mathbb{Z}$ by
	\begin{equation}\label{transition prob}
	\forall z \in \mathbb{Z}, \quad	p(z) = c_{\gamma} |z|^{-\gamma-1}   \mathbbm{1}_{z \neq 0}, 
	\end{equation}
where $\gamma \in (0,2)$ is fixed and $c_{\gamma}$ is a normalizing constant  that turns $p$ into a probability. The jump rate from $x$ to $y$ is now given by 
$p(y-x)\xi_{x,y}(\eta)\tilde c_{y,x}(\eta)$, where $\xi_{x,y}=\eta(x)[1-\eta(y)]$, that  corresponds to the exclusion dynamics and $\tilde{c}_{x,y}(\eta):=\eta(x-1) + \eta(x+1) + \eta(y-1) + \eta(y+1)$ so that the jump rate is null if there are no particles in the vicinity of the sites where the exchange takes place. From this it follows that a jump is only possible if there are at least $m=2$ particles in the vicinity of the exchanging particles. Nevertheless, when jumps are to nearest-neighbors, particles can jump independently of the number of particles in the neighboring sites, see Remark \ref{remsep}. For a scheme on the possible jumps,  see the figure below.

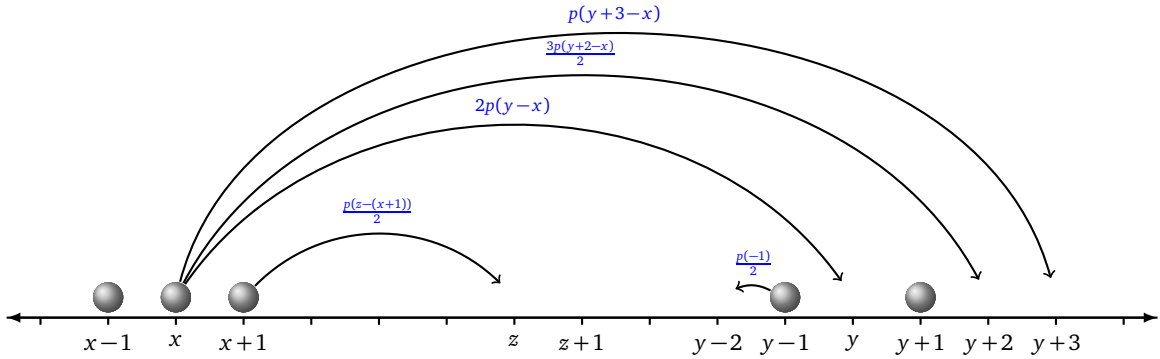
\begin{figure}[htb]
	\begin{center}
		\begin{tikzpicture}[thick, scale=0.9]
			\draw[latex-] (-8.5,0) -- (8.5,0) ;
			\draw[-latex] (-8.5,0) -- (8.5,0) ;
			\foreach \x in  {-8,-7,-6,-5,-4,-3,-2,-1,0,1,2,3,4,5,6,7,8}
			\draw[shift={(\x,0)},color=black] (0pt,0pt) -- (0pt,-3pt) node[below] 
			{};

			\draw[] (-7,-0.1) node[below] {\footnotesize{$x-1$}};
			\draw[] (-6,-0.1) node[below] {\footnotesize{$x$}};
			\draw[] (-5,-0.1)  node[below] {\footnotesize{$x+1$}};
			\draw[] (-1,-0.1)  node[below] {\footnotesize{$z$}};
			\draw[] (0,-0.1)  node[below] {\footnotesize{$z+1$}};
			\draw[] (2,-0.1) node[below] {\footnotesize{$y-2$}};
			\draw[] (3,-0.1) node[below] {\footnotesize{$y-1$}};
			\draw[] (4,-0.1) node[below] {\footnotesize{$y$}};
			\draw[] (5,-0.1)  node[below] {\footnotesize{$y+1$}};
			\draw[] (6,-0.1)  node[below] {\footnotesize{$y+2$}};
			\draw[] (7,-0.1)  node[below] {\footnotesize{$y+3$}};

			\node[shape=circle,minimum size=0.5cm] (F) at (7,0.3) {};
			\node[shape=circle,minimum size=0.5cm] (G) at (6,0.3) {};
			\node[shape=circle,minimum size=0.5cm] (H) at (4,0.3) {};
			\node[shape=circle,minimum size=0.5cm] (I) at (2,0.3) {};
			\node[shape=circle,minimum size=0.5cm] (J) at (-1,0.3) {};

			\node[ball color=black!30!, shape=circle, minimum size=0.4cm] (A) at (-7,0.3) {};
			\node[ball color=black!30!, shape=circle, minimum size=0.4cm] (B) at (-6,0.3) {};
			\node[ball color=black!30!, shape=circle, minimum size=0.4cm] (C) at (-5.,0.3) {};
			\node[ball color=black!30!, shape=circle, minimum size=0.4cm] (D) at (3,0.3) {};
			\node[ball color=black!30!, shape=circle, minimum size=0.4cm] (E) at (5.,0.3) {};

			\path [->] (B) edge[bend left =75] node[above] {\tiny\textcolor{blue}{$p(y+3-x)  $}} (F);           
			\path [->] (B) edge[bend left=62] node[above] {\tiny\textcolor{blue}{$\frac{3p(y+2-x)}{2}$}} (G);        
			\path [->] (B) edge[bend left=55] node[above]  {\tiny\textcolor{blue}{$2p(y-x)$}} (H);
			\path [->] (C) edge[bend left=45] node[above]  {\tiny\textcolor{blue}{$\frac{p(z-(x+1))}{2}$}} (J);
			\path [->] (D) edge[bend right=25] node[above]  {\tiny\textcolor{blue}{$\frac{p(-1)}{2} $}} (I);
			
		\end{tikzpicture}
		\bigskip
		\caption{The fractional porous medium model.}\label{fig.1}
	\end{center}
\end{figure}

The dynamics conserves one quantity: the number of particles in the system. From Remark \ref{remsep}, we can always perform jumps of size $1$, such as the one from $y-1$ to $y-2$ in Figure \ref{fig.1}. This is useful to avoid blocked configurations and assures that our process is irreducible, i.e. we can always move a particle from a position $x$ to any other position $y$ by performing a finite number of jumps with a strictly positive probability. We observe that the system just defined  has degenerate rates in the sense that, jumps of size strictly greater than two can have rate equal to zero (as long as there are no particles in the vicinity of the sites where the exchange occurs), nevertheless jumps of size one have always a strictly positive probability to occur. As in the nearest-neighbor case described above, the Bernoulli product measures with constant parameter are invariant and repeating the same heuristic computations as we did above, in this case, since we allow long jumps, we obtain in the limit the porous medium equation where we replace the Laplacian operator by a fractional Laplacian operator. Our result then says that the space-time evolution of the density of particles is given by the fractional porous medium equation written in \eqref{eq:porosos}. More precisely, the empirical measure associated with the density of particles, converges to a deterministic measure, which is absolutely continuous with respect to the Lebesgue measure and its density is the unique weak solution of the fractional porous medium equation. Our proof follows the entropy method of \cite{GPV} which gives for free existence of weak solutions, and on the way we prove uniqueness of those weak  solutions.

We note that the cornerstone in our proof is the entropy method, nevertheless, its application is not straightforward and now we explain the difficulties that we faced along our proofs. The starting point is Dynkin's formula, see \eqref{dynkin} which gives the control of the boundary terms in time of  the empirical measure  as  the sum of the integral relying on  the action of the generator plus a martingale that controls the   noise. This formula provides a discretization of what we expect to obtain at the macroscopic level, i.e. the notion of the weak solution to the PDE is obtained by taking the limit $n \rightarrow  \infty$ of \eqref{dynkin}, where $n$ is the scaling parameter. We obtain a deterministic equation and not a stochastic one, because the martingale will vanish as $n \rightarrow  \infty$. We note that  since  our model is evolving on $\mathbb Z$ we need to be very careful in all our estimates. Moreover, our rates depend on the configuration in other sites close to those where the exchange takes place, therefore, to close the Dynkin's formula in terms of the empirical measure we need several replacement lemmas. The idea behind these lemmas is that one should replace products of $\eta's$ by products of averages in big microscopic boxes and this  average corresponds to the empirical measure evaluated on a certain function, closing therefore the equation. The limit of this average gives exactly the profile which is the solution (in the weak sense) of our PDE.

We highlight that the study of fractional PDEs from interacting particle systems is quite recent, apart the article \cite{jara2009hydrodynamic}. Recently, the fractional heat equation has been derived with several boundary conditions either on the interval $[0,1]$ with a slow/fast boundary  (see \cite{byronsdif, stefano}) or on $\mathbb Z$ with a slow barrier (see \cite{CGJ2}).  But all the fractional equations are linear, as well as, the boundary conditions. 
The rigorous study of  nonlinear versions of those equations has given many challenges from the mathematical point of view, since at the same time one has to  treat the nonlinearity and the fractional diffusion. From the particle system, one has to deal with blocked configurations and degenerate rates but also  the  nonlocality of the exchange rate of particles.

As variations of our dynamics we note that we could also consider the analogous cases as described above for the exclusion on the interval $[0,1]$ with a slow/fast boundary; or on $\mathbb Z$ with a slow barrier. In these cases we would obtain the fractional porous medium equation with boundary conditions of Dirichlet, Robin or Neumann type. It would also be very interesting to extend the recent results of \cite{GNS} to obtain the fractional porous medium equation above but for any real power $m>0$, including the fast super-diffusion case i.e. when $m\in(0,1)$.  All this is left for future work.

Here follows an outline of this article. 
In Section \ref{definition of the model} we introduced our model, we present the notion of weak solution to the porous medium equation and we state our main result, namely the \textit{hydrodynamic limit}.  In Section \ref{sec:heuristics} we present an heuristic argument used to deduce the porous medium equation as the hydrodynamic equation.
In Section \ref{sectight} we prove tightness of the sequence of empirical measures associated with the density of particles.  From that section we know that the sequence of empirical measures has weakly converging subsequences. Section \ref{secchar} deals with the characterization of those limit points. In Section \ref{secreplem} we prove all the technical results which allow dealing with the non-linearity at the microscopic level. In the appendices we collect all the auxiliary results that are needed along the proofs. 

\newpage

\section{Statement of results}\label{definition of the model}

\subsection{The model}
In order to properly define our results we  begin by describing the microscopic dynamics considered in this article. Let $\mathbb{Z}$ be the set of integer numbers. Our state space is the set $\Omega = \{0,1\}^{\mathbb{Z}}$ and we call its elements \textit{configurations}, which are denoted by Greek letters $\eta,\xi$. The elements of $\mathbb{Z}$ are called sites and are denoted by Latin letters $x,y,z$. Given a configuration $\eta \in \Omega$ and a site $x \in \mathbb{Z}$, we say that the site $x$ is empty if $\eta(x)=0$, and that the site $x$ is occupied if $\eta(x)=1$. Our particles will move between sites in $\mathbb Z$ and according to a probability measure defined in \eqref{transition prob}. %  $\mathbb{Z}$  by 
%\begin{equation}\label{transition prob}
%p(z) = c_{\gamma} |z|^{-\gamma-1}   \mathbbm{1}_{z \neq 0}, \forall z \in \mathbb{Z},
%\end{equation}
%where $\gamma \in (0,2)$ is fixed and $c_{\gamma}$ is a normalizing constant. 
Given an initial configuration $\eta$, after the exchange of particles  between the sites $x$ and $y$ the new configuration will be denoted by $\eta^{x,y}$, where 
\begin{equation*}
\eta^{x,y}(z) := 
\begin{cases}
\eta(z), \; z \ne x,y,\\
\eta(y), \; z=x,\\
\eta(x), \; z=y.
\end{cases}
\end{equation*}
Unless it is explicitly stated otherwise, all the discrete variables in the summations below will range over $\mathbb{Z}$.   We say that  $f: \Omega \rightarrow  \mathbb{R}$ is a \textit{local function}, if there exists a finite $\Lambda \subset \mathbb{Z}$ such that  
$
\forall x \in \Lambda, \eta(x) = \xi(x)  \Rightarrow f(\eta) = f(\xi).
$ Our continuous time Markov process $(\eta_{t})_{t\geq 0}$ is characterized by its infinitesimal generator $\mcb L$ which is given on local functions $f: \Omega \rightarrow  \mathbb{R}$ by 
\begin{equation} \label{gennotsped}
(\mcb L f)(\eta) = \frac{1}{4}\sum_{x,y } p(y-x)c_{x,y}(\eta)[f(\eta^{x,y})-f(\eta)],
\end{equation}
where 
\begin{equation}  \label{rates}
c_{x,y}(\eta) :=  \tilde{c}_{x,y}(\eta)   \xi_{x,y}(\eta),
\end{equation}
with
\begin{align} \label{eq:rates_porous}
\tilde{c}_{x,y}(\eta):=\eta(x-1) + \eta(x+1) + \eta(y-1) + \eta(y+1)
\end{align}
and
$
\xi_{x,y}(\eta):=\eta(x) [1 - \eta(y)] + \eta(y) [ 1 - \eta(x) ].
$
\begin{rem} \label{remsep}
Choosing $y=x+1$ in \eqref{rates} we get
\begin{align*}
c_{x,x+1}(\eta)=& [ \eta(x-1) + \eta(x+1) + \eta(x) + \eta(x+2) ]  \big(  \eta(x) [1 - \eta(x+1)] + \eta(x+1) [ 1 - \eta(x) ] \big) \\
=&\xi_{x,x+1}(\eta) [ \eta(x-1) + \eta(x+2) + 1].
\end{align*}
Above we used the fact that $\eta(z) \in \{0,1\}$ for every $z \in \mathbb{Z}$. In particular, our dynamics always allows jumps of size $1$ avoiding blocked configurations.
\end{rem}

\subsection{Empirical measure}

Hereafter we fix $T > 0$ and a finite time horizon $[0,T]$. We consider the Markov process speeded up by the time scale $n^{\gamma}$; in this way, we denote $\eta_t^n:=\eta_{t n^{\gamma}}$ and observe that the infinitesimal generator of $(\eta_t^n)_{t \in [0,T]} $ is $n^{\gamma} \mcb L$.  Let us now define the empirical measure associated to the density in this process. For $\eta \in \Omega$, this measure gives weight $1/n$ to each particle in the following way:
\begin{equation*}
\pi^{n}(\eta, du) := \frac{1}{n}\sum_{x }\eta(x)\delta_{x/n}(du),
\end{equation*}
where $\delta_u$ is a Dirac mass on $u \in \mathbb{R}$. In order to analyse the temporal evolution of the empirical measure, we define the process of the empirical measures as $\pi^{n}_{t}(\eta,du) := \pi^{n}(\eta^n_{t},du)$. For a test function $G:\mathbb{R} \rightarrow \mathbb{R}$, we denote the integral of $G$ with respect to the empirical measure $\pi_{t}^{n}$, by $\langle \pi^{n}_{t},G \rangle$. We note that this notation should not be mixed up with the inner product in $L^2(\mathbb R)$ that we will introduce below.
For $t \in [0,T]$, we observe that $\pi^{n}_{t} \in \mcb {M}^+$, where $  \mcb {M}^+$ is the space of non-negative Radon measures on $\mathbb{R}$ and equipped with the weak topology.

\subsection{Fractional porous medium equation}
Let $g:\mathbb{R} \rightarrow [0,1]$ and $\rho:[0,T]\times\mathbb{R} \rightarrow [0,1]$. We are interested in deriving the fractional porous medium equation given by
\begin{equation}\label{fractionalPME}
\begin{cases}
&\partial_{t}\rho(t,u) =  [-(-\Delta)^{\gamma /2} \rho^2](t,u) , \;\; u\in \mathbb{R}, \; t \in [0,T], \\
&\rho(0,u) = g(u), \;\; u\in \mathbb{R}.
\end{cases}
\end{equation}
Above, the fractional Laplacian $-(-\Delta)^{\gamma/2}$ of exponent $\gamma/2$ is defined on the set of functions $G:\mathbb{R} \rightarrow \mathbb{R}$ such that
\begin{align*}
\int_{\mathbb{R}} \frac{G(u)}{(1+|u|)^{1+\gamma}} du < \infty
\end{align*}
 by
\begin{equation}\label{eq:frac_lap}
[-(-\Delta)^{\gamma/2}G](u) := c_{\gamma} \lim_{\varepsilon \rightarrow 0^{+}} \int_{\mathbb{R}} \mathbbm{1}_{\{ |u-v|\geq \varepsilon \} }\frac{G(v)-G(u)}{|u-v|^{1+\gamma}}\; dv
\end{equation}
provided the limit exists.  Above, $c_\gamma$ is the constant appearing in \eqref{transition prob}. 
We note that an equivalent definition for the fractional Laplacian given in last display  is through the Fourier transform, i.e. $\widehat{-(-\Delta)^{\gamma/2}}G(\xi)=|\xi|^\gamma\widehat G(\xi)$, nevertheless, we will not use this definition in this article.
\begin{definition}
The Sobolev space $\mcb {H}^{\gamma/2}$ in $\mathbb{R}$ consists of all functions $f \in L^2(\mathbb{R})$ such that 
\begin{align*}
\iint_{\mathbb{R}^2} \frac{[f(u)-f(v)]^2}{|u-v|^{1+\gamma}} du dv < \infty. 
\end{align*}
This is a Hilbert space for the norm $\| \cdot \|_{\mcb {H}^{\gamma/2}}$ defined by
	\begin{align*}
		\| f \|^{2}_{\mcb {H}^{\gamma/2}} := \int_{\mathbb{R}} [f(u)]^2 du + \iint_{\mathbb{R}^2} \frac{[f(u)-f(v)]^2}{|u-v|^{1+\gamma}} du dv .
	\end{align*}
\end{definition}
Below we use the notation $\langle f,g\rangle$ to denote the inner product between two functions $f,g\in  L^2(\mathbb R)$. 
Below $N \subset L^2(\mathbb{R})$ is a metric space with norm $\| \cdot \|_{N}$.
\begin{definition}
The space $L^{2}\left(0,T; N \right)$ is the set of all measurable functions $f:[0,T] \times \mathbb{R} \rightarrow \mathbb{R}$ such that $f(s, \cdot) \in N$ for almost every $s$ on $[0,T]$ and
$\int_{0}^{T}\| f(s, \cdot) \|^{2}_{N}\; ds < \infty .$ Moreover, the set $P ([0,T], N )$ is the space of functions $G: [0,T] \times \mathbb{R} \rightarrow \mathbb{R}$ such that there exist $k \in \mathbb{N}:=\{0, 1, 2, \ldots, \}$ and $G_0, G_1, \ldots, G_k \in N$ so that 
	\begin{equation}\label{space_functions}
		\forall (t,u) \in [0,T] \times \mathbb{R}, \quad G(t,u) = \sum_{j=0}^{k} t^{j} G_j(u).
	\end{equation}
\end{definition}

Given $r \in \{1, 2, \ldots\}$,  $G: \mathbb{R} \rightarrow \mathbb{R}$ is in $C^{r}(\mathbb{R})$ if $G$ is $r$ times continuously differentiable and for $r=0$, $C^0(\mathbb{R})$ denotes the set of continuous functions in $\mathbb{R}$. Also, $G \in C_{c}^{r}(\mathbb{R})$ if $G \in C^{r}(\mathbb{R})$ and $G$ has compact support. Moreover, we use the notation $C_c^{\infty}(\mathbb{R}):=\cap_{r=0}^{\infty} C_{c}^r (\mathbb{R})$.

Our space of test functions is $\mcb S:= P \big( [0,T], C_c^{\infty}(\mathbb{R}) \big)$. For every $G \in \mcb S$, we denote
\begin{equation} \label{defbG}
	b_{G}:=\inf \Big\{ \forall s \in [0,T], \quad \tilde{b} \geq 0: \sup_{|u| \geq \tilde{b} }|G(s,u)|=0  \Big\}. 
\end{equation}
Finally, for every bounded function $G: X \rightarrow \mathbb{R}$, we denote $\| G \|_{\infty}:=\sup_{u \in X} |G(u)|$.

\subsection{The main result}

For every $n \geq 1$, let $\mu_n$ be a probability measure on $\Omega$, which is the space of configurations. Next, let    $\mathbb{P}_{\mu_n}$ be the probability measure on the Skorokhod space $\mcb {D}([0,T],\Omega)$ induced by the  Markov process $(\eta^n_{t})_{t \in [0,T]}$ and the initial measure $\mu_n$; this is a measure on the space of trajectories of configurations. Finally, let $\mathbb{Q}_{n}$ be the probability measure on $\mcb{D}([0,T],  \mcb {M}^{+})$ induced by $(\pi_{t}^{n})_{t \in [0,T]}$ and $\mathbb{P}_{\mu_n}$; this is a measure on the space of  trajectories of measures.

\begin{definition}\label{associated profile}
Let $g: \mathbb{R}\rightarrow[0,1]$ be a measurable function and $(\mu_n )_{n\ \geq 1}$ a sequence of probability measures in $\Omega$. We say that $(\mu_n )_{n\ \geq 1}$ is \textit{associated with $g(\cdot)$}, if for any $G \in C_c^0(\mathbb{R})$ and any $\delta > 0$, 
\begin{equation*}
\lim _{n \rightarrow \infty } \mu_n \Big( \eta \in \Omega : \Big|  \frac{1}{n} \sum_{x} G\Big(\tfrac{x}{n}\Big) \eta(x)   - \int_{\mathbb{R}} G(u)g(u)\,du \, \Big|    > \delta \Big)= 0.
\end{equation*} 
\end{definition}
 We observe that the previous definition is simply requiring a weak convergence of the random measure $\pi^n_0$ to a deterministic one, i.e. to $g(u)du$. The goal in hydrodynamic limits is to show that the previous result is true at any time $t$ where the density of the limiting measure is a weak solution to a PDE, called the hydrodynamic equation. 
Now we define the notion of weak solution of the hydrodynamic equation that we obtain. 
\begin{definition}\label{eq:dif}
Let $g: \mathbb{R} \rightarrow [0,1]$ be a measurable function. We say that $\rho :[0,T] \times \mathbb{R} \rightarrow [0,1]$ is a weak solution of the fractional porous medium equation in $\mathbb{R}$ with initial condition $g$ 
\begin{equation} \label{fracpor}
\begin{cases}
\partial_t \rho(t,u) = [-(- \Delta)^{\gamma/2} \rho^2 ] (t,u), & (t,u) \in [0,T] \times \mathbb{R}, \\
\rho(0,u) = g(u), & u \in \mathbb{R}, 
\end{cases}
\end{equation}
if the following conditions hold:
\begin{enumerate}
\item
for every $t \in [0,T]$ and for every $G \in \mcb S$, it holds  $F(t, \rho,G,  g)=0$, where
\begin{align*}
F(t, \rho,G,  g):= & \langle \rho_t, G_t\rangle - \langle  g, G_0\rangle - \int_0^t \langle \rho_s, \partial_s G_s\rangle ds- \int_0^t \langle \rho^2_s,  [-(- \Delta)^{\gamma/2} G_s]\rangle ds; 
\end{align*}
\item 
there exists $b \in (0,1)$ such that $ \rho - b \in L^2 \big(0, T ; L^{2} ( \mathbb{R} ) \big)$ and $ \rho^2 - b^2 \in L^2 (0, T ; \mcb{H}^{\gamma/2} )$.
\end{enumerate}
\end{definition}
\begin{rem}
The uniqueness of weak solutions of \eqref{fracpor} is proved in Appendix \ref{secuniq}. 
\end{rem}
\begin{definition} \label{defmub}
Given $b \in (0,1)$, we define the measure $\nu_b$ on $\Omega$, which is the Bernoulli product measure with marginals $\nu_b\{ \eta \in \Omega: \eta(y)= 1\}=b$, for every $y \in \mathbb{Z}$.
\end{definition}
\begin{rem} \label{revmeas}
We observe that under $\nu_b$, the random variables $(\eta(y))_{y \in \mathbb{Z}}$ are i.i.d. with Bernoulli distribution of parameter $b$.  Moreover, it holds $
	\nu_b (\eta^{x,y}) = \nu_b (\eta)$, for every $\eta \in \Omega$ and for every $x,y \in \mathbb{Z}$.
Combining last identity with the symmetry of $p$ given in \eqref{transition prob}, we can conclude that  $\nu_b$ is a reversible measure with respect to  $n^{\gamma} \mcb L$. 
\end{rem}
Finally, we can state the main result of this article. Hereafter we say that $g(n) \lesssim h(n)$ if there exists $C>0$ such that $|g(n)| \leq C |h(n)|$, for every $n \geq 1$. 
\begin{thm}\label{hydlim}
Let $g: \mathbb{R} \rightarrow [0,1]$ be a measurable function. Let $(\mu_n)_{n \geq 1}$ be a sequence of probability measures in $\Omega$ associated to the profile $g$ such that there exists  $C_b>0$ such that 
	\begin{equation} \label{defcb}
\forall n \geq 1, \quad		H ( \mu_n | \nu_b) \leq C_b n,
	\end{equation}
 for some $b \in (0,1)$. Then, for any $t \in [0,T]$, any $G \in C_c^0(\mathbb{R})$ and any $\delta>0$,
\begin{equation*}\label{limHidreform}
\lim_{n \rightarrow \infty}\mathbb{P}_{\mu_n}\Big(  \eta_{\cdot}^n \in \mcb {D}([0,T], \Omega): \Big| \frac{1}{n} \sum_{x} G(\tfrac{x}{n}) \eta_t^n(x) - \int_{\mathbb{R}} G(u)\rho(t,u)\,du  \,  \Big| > \delta \Big)=0,
\end{equation*}
where $\rho(t, \cdot)$ is the unique weak solution of \eqref{fracpor}.
\end{thm}
Now we describe the strategy of the proof of Theorem \ref{hydlim}. We follow the entropy method introduced in \cite{GPV}. In Section \ref{sectight}, we prove that $(\mathbb{Q}_{n})_{n \geq 1}$ is tight with respect to the Skorokhod topology of $\mcb {D}([0,T],  \mcb{M}^{+})$ and therefore due to Prohorov's Theorem (see Theorem 6.1 in \cite{Bill}), it is relatively compact. This means that $(\mathbb{Q}_{n})_{n \geq 1}$ has a convergent subsequence, i.e., there exists a subsequence  $(\mathbb{Q}_{n_j})_{j \geq 1}$ and a measure $\mathbb{Q}$ such that $(\mathbb{Q}_{n_j})_{j \geq 1}$ converges (weakly) to $\mathbb{Q}$, and from here on we will refer to  $\mathbb{Q}$ as a \textit{limit point}.  In Subsection \ref{sec_4.1} (resp. Subsection \ref{sec_4.2}) we prove that any such limit point $\mathbb{Q}$ is concentrated on trajectories of measures satisfying the first   (resp. the second) condition of weak solutions of \eqref{fracpor}. Combining this with the uniqueness of weak solutions of \eqref{fracpor} (proved in Appendix \ref{secuniq}), we can conclude that the aforementioned limit point $\mathbb{Q}$ is actually unique, leading to the conclusion of Theorem \ref{hydlim}. Some auxiliary replacement lemmas and discrete convergences are proved in Section \ref{secreplem} and Appendix \ref{secdiscconv}, respectively, and in Appendix \ref{ap:prop} we present some properties of the fractional Laplacian, while in Appendix \ref{secuniq}
 we prove the uniqueness of weak solutions to the porous medium equation. 
 
\section{Heuristic argument to deduce the hydrodynamic equation}\label{sec:heuristics}

In this section we present an heuristic argument that allows us to derive the integral equation in \eqref{fracpor}. We assume by now that the sequence $(\mathbb{Q}_{n})_{n \geq 1}$ is tight  (this fact will be proved in the next section) and let $\mathbb{Q}$ be a limit point. 
A simple computation based on the fact that our variables are bounded, allows showing that the limit  measure $\mathbb{Q}$ is concentrated on trajectories of measures $\pi_t(du)$ that are  absolutely continuous with respect to the Lebesgue measure, i.e. $\pi_t(du):=\rho_t(u)du$. Now we need to characterize $\rho_t(u)$ as a weak solution to the fractional  porous medium equation.  According to Dynkin's formula (see Lemma $A.1.5.1$ of \cite{kipnis1998scaling}), we have that
\begin{equation}\label{dynkin}
\mcb M_{t}^{n}(G) :=\langle  \pi^{n}_{t}, G_t  \rangle - \langle \pi^{n}_{0}, G_0 \rangle - \int_{0}^{t} \partial_s \langle \pi^{n}_{s},G_s \rangle\,ds - \int_{0}^{t}  n^{\gamma} \mcb L\langle \pi^{n}_{s},G_s \rangle\,ds
\end{equation}
is a martingale with respect to the natural filtration $\mcb {F}^n_{t} := \left\{ \sigma(\eta_s^n): s\leq t \right\}$, for every $n \geq 1$, $t \in [0,T]$ and $G \in \mcb{S}$. Since the sequence $(\mathbb Q_n)_{n\geq 1}$ is tight let $n_j$ be a subsequence such that $(\mathbb Q_{n_j})_{j\geq 1}$ weakly converges to $\mathbb Q$, which is supported on trajectories of the form $\pi_t(du)=\rho_t(u)du$. To make notation simple we assume that $n_j=n$. From this it follows that the first three terms on the right-hand side of  \eqref{dynkin} converge, as $n \rightarrow \infty$, in $L^1 (\mathbb{P}_{\mu_n} )$ to
\begin{align*} 
\int_{\mathbb{R}} \rho_t(u) G_t (u) du  - \int_{\mathbb{R}} \rho_0(u)G_0(u) du   -  \int_0^t   \int_{\mathbb{R}} \rho_s(u)  \partial_s G_s(u) du   ds.
\end{align*}
From Definition \ref{associated profile}, we get that $\int_{\mathbb{R}} [\rho_0(u) - g(u)] G_0(u) du$ converges to zero in $L^1 (\mathbb{P}_{\mu_n} )$, as $n \rightarrow \infty$. Hence last display converges in $L^1 (\mathbb{P}_{\mu_n} )$ to
% replace last expression by
\begin{align*}
\int_{\mathbb{R}} \rho_t(u) G_t(u) du  - \int_{\mathbb{R}} g(u) G_0(u) du   -  \int_0^t   \int_{\mathbb{R}} \rho_s(u)  \partial_s G_s(u) du   ds,
\end{align*}
as $n \rightarrow \infty$. Now we focus on last term of \eqref{dynkin}, which is known in the literature as the \textit{integral term}. This term describes the action of the infinitesimal generator in the empirical measure associated to the conserved quantity: the density of particles. This term  will lead us to the fractional porous medium equation. By performing some algebraic manipulations, for every $G \in \mcb S$, it holds 
\begin{align}
 \int_{0}^{t}  n^{\gamma} \mcb L\langle \pi^{n}_{s},G_s\rangle\,ds =& \int_{0}^{t} \frac{1}{2n}\sum_{x }n^{\gamma}\mcb{K}_{n}G_s(\tfrac{x}{n}) \eta_s^n(x) [ \eta_s^n(x-1) + \eta_s^n(x+1) ] ds \label{terprinc} \\
 +& \int_{0}^{t} \frac{n^{\gamma}}{2n} \mcb {R}_{n}^G(s) ds \label{terext},
\end{align}
where $\mcb{K}_{n}$ and $\mcb{R}_{n}$ are defined on functions $G \in \mcb{S}$ as
\begin{equation}\label{defKn}
\mcb{K}_{n}G_s(\tfrac{x}{n}) :=\sum_{y}\left[G_s(\tfrac{y}{n})-G_s(\tfrac{x}{n})\right]p(y-x),
\end{equation}
\begin{equation}\label{defR}
\mcb{R}_{n}^G(s) = \sum_{x} \eta_s^n(x) \Big[      \sum_{y} [ G_s( \tfrac{x+1}{n}) - G_s( \tfrac{x}{n}) + G_s( \tfrac{y}{n}) - G_s( \tfrac{y+1}{n}) ] \eta_s^n(y+1)  p (y-x) \Big].   
\end{equation}
In a nutshell, the argument finishes by noting that the action of the infinitesimal generator in the empirical measure gives rise to the two last terms, and we proceed as follows. First, we  will show that the term with $\mcb{R}_{n}^G(s)$ will be negligible in the limit. Second, in the  remaining term  the discrete operator $\mcb{K}_{n}$ will give rise to the fractional Laplacian (since we are taking the time scale $n^\gamma$; any other time scale less than $n^\gamma$ would give rise to a trivial evolution, since this term would also vanish in the limit); while the terms with the products of $\eta
$'s will give rise to the square of the profile. And this finishes the argument. 
To make the argument more clear we note that in order to treat \eqref{terext}, we  use the  next result, which is proved in Appendix \ref{secdiscconv}.
\begin{prop} \label{convext}
For every $\gamma \in (0,2)$, define $\delta_{\gamma}$ by $\delta_{\gamma}=0$ for $\gamma \in (0,1)$, $\delta_{\gamma}=1/2$ for $\gamma=1$ and $\delta_{\gamma}=1$ for $\gamma \in (1,2)$. Then for every $G \in \mcb S$ it holds 
\begin{align*}
 \frac{1}{n}   \sum_{x,y }  \sup_{s \in [0,T]} n^{\gamma}\Big| G_s(\tfrac{x+1}{n}) - G_s(\tfrac{x}{n}) + G_s(\tfrac{y}{n}) - G_s(\tfrac{y+1}{n}) \Big| p(y-x) \lesssim \max \Big\{ n^{\gamma-2} , n^{-1} , n^{\gamma-1-\delta_{\gamma}}\Big\}. 
\end{align*}
\end{prop}
Combining last proposition with the fact that $| \eta_s^n(\cdot) | \leq 1$, the term in  \eqref{terext} converges to zero in $L^1 (\mathbb{P}_{\mu_n} )$, as $n \rightarrow \infty$. It remains to treat \eqref{terprinc}; we do so by applying next result, which, as we mentioned above,  motivates the choice $n^{\gamma}$ for the time scale. Since it is stated and proved in Proposition A.1 of \cite{CGJ2}, we omit its proof.
\begin{prop} \label{convdisc}
For every $\gamma \in (0,2)$ and $G \in \mcb S$, it holds
\begin{align*}
\lim_{n \rightarrow \infty} \frac{1}{n}   \sum_{x }  \sup_{s \in [0,T]}\Big|n^{\gamma} \mcb{K}_n G_s \left(\tfrac{x}{n} \right)  -[-(- \Delta)^{\gamma/2}   G_s]  \left(\tfrac{x}{n} \right)\Big | =0.
\end{align*}
\end{prop} 
Next, we state another result which is classical, but we did not find its proof in the literature, therefore, we  present it in Appendix \ref{ap:prop}.
\begin{prop} \label{propL1}
The fractional Laplacian maps $G \in C_c^{\infty}(\mathbb{R})$ into $L^1(\mathbb{R}) \cap L^{\infty}(\mathbb{R})$. In particular,

	\begin{equation} \label{L1frac}
\forall G \in \mcb S, \quad		\frac{1}{n} \sum_{x} \sup_{s \in [0,T]} \Big|  [ -(-\Delta)^{\gamma/2}G_s]( \tfrac{x}{n})\Big| < \infty.
	\end{equation}
\end{prop}
At last, in  \eqref{terprinc}  each of the  terms with  $ \eta_s(x)$ will be replaced by an average in a box of microscopic size $\varepsilon n$, which then corresponds to $\langle  \pi^{n}_{s}, \varepsilon^{-1}\mathbbm{1}_{\big[\tfrac{x}{n},\tfrac{x}{n} + \varepsilon \big )}  \rangle$, and this converges as $n \rightarrow  \infty$ and $\varepsilon \rightarrow 0$ to $\rho_s(x/n)$. Since we have products of two $\eta$'s we will obtain  $\rho_s^2$ in the equation. 
Finally, by combining Propositions \ref{convdisc} and \ref{propL1}, we conclude that  \eqref{terprinc} converges to 
\begin{align*}
\int_0^t \langle \rho^2_s,  [-(- \Delta)^{\gamma/2} G_s]\rangle ds
\end{align*}
in $L^1 (\mathbb{P}_{\mu_n} )$, as $n \rightarrow \infty$, leading to the integral equation in \eqref{fracpor}.

\section{Tightness} \label{sectight}
In this section, our goal is to prove that the sequence of probability measures $(\mathbb{Q}_{n})_{n \geq 1}$ is tight with respect to the Skorokhod topology of $\mcb {D}([0,T],  \mcb{M}^{+})$. Following Propositions 4.1.6 and 4.1.7 of \cite{kipnis1998scaling}, in order to prove tightness of $(\mathbb{Q}_{n})_{n \geq 1}$ it is enough to show that
\begin{equation}\label{tight}
\lim_{\omega \rightarrow 0} \limsup_{n \rightarrow \infty} \sup_{\tau_1, \tau_2 \in \mcb{T}_{T}, \tau_2 \leq \omega}\mathbb{P}_{\mu_{n}} \Big(  \eta_{\cdot}^n \in \mcb{D}\left([0,T], \Omega \right): \big| \langle \pi_{\tau_1 + \tau_2}^{n},G \rangle - \langle \pi_{\tau_1}^{n},G \rangle \big| >\varepsilon \Big)=0,
\end{equation}
for every $G \in C_c^{\infty}(\mathbb{R})$ (not depending on time, but we make the presentation more general) and every $\varepsilon > 0$. Above, $\mcb{T}_{T}$ is the set of stopping times bounded by $T$, therefore $\tau_1+\tau_2$ must be read as $\min\{\tau_1 + \tau_2, T\}$. In order to do this, we use Lemma $A.1.5.1$ of \cite{kipnis1998scaling}, which gives that
\begin{align}
&\mathbb{E}_{\mu_n}\Big[ \big( \mcb M^{n}_{\tau_1+\tau_2}(G)- \mcb M_{t_1}^{n}(G) \big)^2 \Big] = \mathbb{E}_{\mu_n}\Big[ \int_{\tau_1}^{\tau_1+\tau_2} n^{\gamma}\big[ \mcb L\langle \pi_{s}^{n},G_s\rangle^{2} -2\langle \pi_{s}^{n}, G_s\rangle \mcb L\langle \pi_{s}^{n}, G_s \rangle \big] ds \Big]. \label{quadratic}
\end{align}
Above and in what follows, $\mathbb{E}_{\mu_n}$ denotes the expectation with respect to $\mathbb{P}_{\mu_{n}}$. Above, $\mcb M_{t}^{n}(G)$ is given in \eqref{dynkin}. By combining \eqref{dynkin} with Markov's and Chebyshev's inequalities, \eqref{tight} is bounded from above by
\begin{align*}
\lim_{\omega \rightarrow 0} \limsup_{n \rightarrow \infty} \sup_{\tau_1, \tau_2 \in \mcb{T}_{T}, \tau_2 \leq \omega} \Big\{\frac{4}{\varepsilon^2} \mathbb{E}_{\mu_{n}}\Big[  |\mcb M_{\tau_1 + \tau_2}^{n}(G) - \mcb M_{\tau_1}^{n}(G) \big|^2 \Big] + \frac{2}{\varepsilon} \mathbb{E}_{\mu_{n}} \Big[ \Big| \int_{\tau_1}^{\tau_1 + \tau_2}  n^{\gamma} \mcb L\langle \pi_{r}^{n},G_r \rangle \, dr \Big| \Big]\Big\}.
\end{align*}
Hence, it is enough to show that last display vanishes  for every $G \in C_c^{\infty}(\mathbb{R})$.
First we analyse the rightmost term in last display. Combining \eqref{terprinc} and \eqref{terext} with Propositions \ref{convdisc} and \ref{convext}, we conclude that for every $G \in C_c^{\infty}(\mathbb{R})$, there exists $C(G)$ such that  $\sup_{s \in [0,T]}  |n^{\gamma} \mcb L\langle \pi^{n}_{s}, G_s  \rangle | \leq C(G)$.
To finish the proof we use the next result together with   \eqref{quadratic}.
\begin{prop} \label{boundquad}
Let $G \in \mcb{S}$. Then
\begin{equation} \label{boundquad1}
\sup_{s \in [0,T]}\Big| n^{\gamma}\left(\mcb L\langle \pi_{s}^{n},G _s \rangle^{2} -2\langle \pi_{s}^{n}, G_s \rangle  \mcb L\langle \pi_{s}^{n}, G_s\rangle \right)\Big| \lesssim \max\{ n^{\gamma-2}, n^{-1} \}. 
\end{equation}
\end{prop}
\begin{proof}
After performing some algebraic manipulations, the expression on the left-hand side of \eqref{boundquad1} can be rewritten as
\begin{align*}
\sup_{s \in [0,T]} \frac{n^{\gamma}}{4n^{2}}\sum_{x,y } [ G_s\left(\tfrac{y}{n}\right)-G_s\left(\tfrac{x}{n}\right) ]^{2} p(x-y)  c_{x,y}(\eta_s^n) [\eta_{s}^n(x)-\eta_{s}^n(y) ]^{2} .
\end{align*}
Moreover, from Proposition A.10 of \cite{CGJ2}, we have that
\begin{align} \label{tight2condaux}
\forall G \in \mcb S, \quad n^{\gamma-2} \sum_{x,y} \sup_{s \in [0,T]} [ G_s ( \tfrac{y}{n} ) - G_s ( \tfrac{x}{n} ) ]^2 p(y - x)  \lesssim \max\{ n^{\gamma-2}, n^{-1} \}.
\end{align}
Combining this with the facts that $|c_{x,y}(\eta_s^n)| \leq 4$ and $[\eta_{s}^n(x)-\eta_{s}^n(y) ]^{2} \leq 1$ for every $s \in [0,T]$, the proof ends.
\end{proof}

\section{Characterization of  limit points} \label{secchar}
From the results of Section \ref{sectight}, we know  that $(\mathbb{Q}_n)_{n \geq 1}$ has at least one limit point $\mathbb{Q}$. From \cite{kipnis1998scaling}, since every site has at most one particle, any limit point $\mathbb{Q}$ is concentrated on trajectories of measures that are absolutely continuous with respect to the Lebesgue measure, i.e., $$\pi_t(du)=\rho(t,u)du$$ for almost every $t$ on $[0,T]$. In this section we will prove  additional properties of $\mathbb{Q}$: it is also concentrated on trajectories such that $\rho$ satisfies the first and second conditions of weak solutions of \eqref{fracpor}.  We start by showing that the first condition is satisfied. 
\subsection{The validity of condition (1) in Definition \ref{eq:dif} }\label{sec_4.1}
\begin{prop}
If $\mathbb{Q}$ is a limit point of $(\mathbb{Q}_{n})_{n \geq 1}$ then

	\begin{equation*}
		\mathbb{Q}\Big(  \pi_{\cdot} \in \mcb{D} \left([0,T], \mcb{M}^{+}\right): \forall t\in[0,T], \forall G \in \mcb{S}, \quad F(t,\rho,G,g)=0   \Big)=1,
	\end{equation*} 
where $F(t,\rho,G,g)$ is given in Definition \ref{eq:dif} .
\end{prop}
\begin{proof}
In order to prove the proposition, it is enough to verify that for any $\delta>0$ and any $G \in \mcb{S}$,
\begin{equation}\label{charact1}
\mathbb{Q} \Big( \pi_{\cdot} \in \mcb{D} \left([0,T], \mcb{M}^{+}\right): \sup_{t \in [0,T]} |F(t,\rho,G,g)|> \delta\Big)=0.
\end{equation}
In order to simplify the notation, we will omit $\pi_{\cdot}$ from the sets where we are looking at. From the definition of $F$, we get $|F(t,\rho,G,g)| \leq |F(t,\rho,G,\rho_0)|+| \langle \rho_0-g,G_0\rangle|$, so that we can bound \eqref{charact1} from above by 
\begin{equation}\label{charact2}
\mathbb{Q} \Big(  \sup_{t \in [0,T]} |F(t,\rho,G,\rho_0)| >\frac{\delta}{2} \Big)+\mathbb{Q} \Big( |\langle \rho_0-g,G_0\rangle |> \frac{\delta}{2} \Big).
\end{equation}
The rightmost term in last display is equal to zero since $\mathbb{Q}$ is a limit point of $(\mathbb{Q}_{n})_{n \geq 1}$ and $\mathbb{Q}_{n}$ is induced by $\mu_n$ which is associated with $g$, see Definition \ref{associated profile}. Next we rewrite the leftmost term in \eqref{charact2} as
\begin{align*}
\mathbb{Q} \Big(  \sup_{t \in [0,T]} \Big| \langle \rho_t, G_t \rangle - \langle \rho_0, G_0 \rangle - \int_0^t \langle \rho_s, \partial_s G_s \rangle ds - \int_0^t\langle (\rho_s)^2,  [ -(-\Delta)^{\gamma/2}G_s] \rangle ds  \Big| >\frac{\delta}{2} \Big).
\end{align*}
Since the set in last probability is not open regarding the Skorohod topology, we make use of some approximations of the identity in order to apply Portmanteau's Theorem (see Theorem 2.1, Chapter 1 in \cite{Bill}). More exactly, for $u \in \mathbb{R}$ fixed, we define the approximations of the identity $\overleftarrow{i_{\varepsilon}^{ u}}$ and $ \overrightarrow{i_{\varepsilon}^{ u }}$ by 
\begin{align*}
\forall v \in \mathbb{R}, \quad \overleftarrow{i_{\varepsilon}^{ u}}(v) := \frac{1}{\varepsilon} \mathbbm{1}_{ [u-\varepsilon,u) }(v), \quad \text{and} \quad  \overrightarrow{i_{\varepsilon}^{ u }}(v) := \frac{1}{\varepsilon} \mathbbm{1}_{ (u, u+\varepsilon] }(v).
\end{align*}
Putting this together with the fact that $\pi_t(du) = \rho(t,u)du$, we get 
\begin{align*}
\langle \pi_s, \overleftarrow{i_{\varepsilon}^{ u}} \rangle = \langle \rho_s, \overleftarrow{i_{\varepsilon}^{ u}} \rangle = \frac{1}{\varepsilon} \int_{u-\varepsilon}^{u} \rho(s,v) dv\quad \textrm{and}\quad
\langle \pi_s, \overrightarrow{i_{\varepsilon}^{ u } }\rangle = \langle \rho_s, \overrightarrow{i_{\varepsilon}^{ u } }\rangle = \frac{1}{\varepsilon} \int_{u}^{u+\varepsilon} \rho(s,v) dv.
\end{align*}
Combining this with the fact that $\rho \in [0,1]$ and Lebesgue's Differentiation Theorem, we conclude that
\begin{equation} \label{aproxid}
\lim_{\varepsilon \rightarrow 0^{+}} | \langle \pi_s, \overleftarrow{i_{\varepsilon}^{ u}} \rangle - \rho(s,u) | = \lim_{\varepsilon \rightarrow 0^{+}} | \langle \pi_s, \overrightarrow{i_{\varepsilon}^{ u } }\rangle - \rho(s,u) | = 0,
\end{equation}
for almost every $u \in \mathbb{R}$. Moreover, from Proposition \ref{propL1}, we have that $[-(-\Delta)^{\gamma/2}G_s] \in  L^1(\mathbb{R})$. Combining this observation with the fact that $\rho \in [0,1]$ and Lebesgue´s Differentiation Theorem, it is enough to show that
\begin{equation} \label{preport}
\begin{split}
\lim_{\varepsilon \rightarrow 0^+} \mathbb{Q} \Big(  \sup_{t \in [0,T]} \Big| \langle \rho_t, G_t \rangle - \langle \rho_0, G_0 \rangle -& \int_0^t \langle \rho_s, \partial_s G_s \rangle ds \\
-& \int_0^t \int_{\mathbb{R}} \langle \rho_s, \overleftarrow{i_{\varepsilon}^{ u}} \rangle \langle \rho_s, \overrightarrow{i_{\varepsilon}^{ u}} \rangle   [ -(-\Delta)^{\gamma/2}G_s](u) du ds  \Big| >\frac{\delta}{4} \Big).
\end{split}
\end{equation}
We still cannot use Portmanteau´s Theorem directly, since the functions $\overleftarrow{i_{\varepsilon}^{ u}}$, $\overrightarrow{i_{\varepsilon}^{ u}}$ and $[ -(-\Delta)^{\gamma/2}G_s]$ are not in $C_c^{0}(\mathbb{R})$. This motivates us to perform two operations: first we use the fact that that $|\rho_s| \leq 1$ to approximate $\rho_s, \overleftarrow{i_{\varepsilon}^{ u}} \rangle$ and $\rho_s, \overrightarrow{i_{\varepsilon}^{ u}} \rangle$ by $\rho_s, g_{1,\varepsilon} \rangle$ and $\rho_s, g_{2,\varepsilon} \rangle$ in a way that the error vanishes when $\varepsilon \rightarrow 0^+$. Afterwards, we approximate $[ -(-\Delta)^{\gamma/2}G] \in L^1([0,T] \times \mathbb{R})$ by a sequence $(H_k)_{k \geq 1} \in C_c^{\infty}( [0,T] \times \mathbb{R} )$. Now, after an application of Portmanteau´s Theorem the display in \eqref{preport} is bounded from above by  
\begin{align*}
\limsup_{\varepsilon \rightarrow 0^+} 
\liminf_{n \rightarrow \infty}\,\,\mathbb{P}_{\mu_n} \Big( \sup_{t \in [0,T]} \Big| \mcb M_{t}^{n}(G) + & \int_{0}^{t} n^{\gamma} \mcb L\langle \pi^{n}_{s},G_s \rangle ds \\
 - & \int_{0}^{t}\int_{\mathbb{R}} \langle \rho_s, \overleftarrow{i_{\varepsilon}^{ u}} \rangle \langle \rho_s, \overrightarrow{i_{\varepsilon}^{ u}} \rangle,  [ -(-\Delta)^{\gamma/2}G_s](u) du  ds \Big| > \frac{\delta}{16} \Big).
\end{align*}
Above we summed and subtracted  $\int_{0}^{t}n^{\gamma} \mcb L\langle \pi_{s}^{n}, G_s\rangle\, ds$  to the term inside the absolute value in \eqref{preport}, and applied \eqref{dynkin} and the definition of $\mathbb{Q}_{n}$. Last display is bounded from above by 
\begin{equation}\label{charact5}
\liminf_{n \rightarrow \infty} \,\,\mathbb{P}_{\mu_n} \Big( \sup_{t \in [0,T]} \left|\mcb M_{t}^{n}(G) \right| > \frac{\delta}{32} \Big)
\end{equation}
\vspace{-4mm}
\begin{equation}\label{charact6}
+ \limsup_{\varepsilon \rightarrow 0^+} \liminf_{n \rightarrow \infty} \,\, \mathbb{P}_{\mu_n} \Big( \sup_{t \in [0,T]} \Big| \int_{0}^{t} \Bigg[ n^{\gamma} \mcb L\langle \pi^{n}_{s},G_s \rangle  -  \int_{\mathbb{R}} \langle \rho_s, \overleftarrow{i_{\varepsilon}^{ u}} \rangle \langle \rho_s, \overrightarrow{i_{\varepsilon}^{ u}} \rangle,  [ -(-\Delta)^{\gamma/2}G_s](u) du \Bigg] ds \Big| > \frac{\delta}{32} \Big).
\end{equation}
From Doob's inequality, Lemma A1.5.1 of \cite{kipnis1998scaling} and Proposition \ref{boundquad}, we conclude that \eqref{charact5} is equal to zero. Now we treat \eqref{charact6}. From \eqref{terprinc} and \eqref{terext}, it can be rewritten as
\begin{align}
\limsup_{\varepsilon \rightarrow 0^{+}} &\liminf_{n \rightarrow \infty} \,\,  \mathbb{P}_{\mu_n} \Big( \sup_{t \in [0,T]} \Big| \int_{0}^{t} \frac{1}{2n}\sum_{x }n^{\gamma}\mcb{K}_{n}G_s( \tfrac{x}{n}) \eta_s^n(x) [ \eta_s^n(x-1) + \eta_s^n(x+1) ]  ds \nonumber \\
+& \int_{0}^{t} \frac{n^{\gamma}}{2n} \mcb {R}_{n}^G(s) ds  - \int_{0}^{t} \int_{\mathbb{R}} \big[ \langle \pi_s, \overleftarrow{i_{\varepsilon}^{ u} }\rangle \cdot \langle \pi_s, \overrightarrow{i_{\varepsilon}^{ u } }\rangle   [ -(-\Delta)^{\gamma/2}G_s](u)   \big]du ds \Big| > \frac{\delta}{32}\Big). \label{charact9}
\end{align}
Since the error from changing the integral in the space variable by its Riemann sum is or order $n^{-1}$, it is enough to prove that
\begin{equation} \label{char95}
\begin{split}
\limsup_{\varepsilon \rightarrow 0^{+}} \liminf_{n \rightarrow \infty} \,\, \mathbb{P}_{\mu_n}\Big( \sup_{t \in [0,T]} \Big| \int_{0}^{t} \Big\{& \frac{1}{2n} \sum_{x}   n^{\gamma} \mcb{K}_n G_s (\tfrac{x}{n}) \eta_s^n(x) [ \eta_s^n(x-1) +  \eta_s^n(x+1)  ]  +\frac{n^{\gamma}}{2n} \mcb {R}_{n}^G(s)  \ \\
-& \frac{1}{2n} \sum_{x}  \langle \pi_s, \overleftarrow{i}_{\varepsilon}^{ \frac{x-1}{n} }  \rangle \cdot \langle \pi_s, \overrightarrow{i}_{\varepsilon}^{ \frac{x-1}{n} } \rangle    [ -(-\Delta)^{\gamma/2}G_s](\tfrac{x-1}{n})  \\
-& \frac{1}{2n} \sum_{x} \langle \pi_s,  \overleftarrow{i}_{\varepsilon}^{ \frac{x}{n} } \rangle \cdot \langle \pi_s, \overrightarrow{i}_{\varepsilon}^{ \frac{x}{n} } \rangle   [ -(-\Delta)^{\gamma/2}G_s](\tfrac{x}{n})    \Big\} ds\Big| > \frac{\delta}{32}\Big)
\end{split}
\end{equation}
is equal to zero.

For $\ell \geq 1$ and $x \in \mathbb{Z}$ we define the empirical averages on a box of size $\ell$ around $x$ as 
\begin{equation} \label{medemp}
\overrightarrow{\eta}^{\ell}(x):=\frac{1}{\ell} \sum_{y=1}^{\ell} \eta(x+y)   \; \; \text{and} \; \; \overleftarrow{\eta}^{\ell}(x):=\frac{1}{\ell} \sum_{y=-\ell}^{-1} \eta(x+y).
\end{equation}
From here on we interpret $\varepsilon n$ as $\lfloor \varepsilon n\rfloor$.
Observe that $|  \langle \pi_s,  \overleftarrow{i}_{\varepsilon}^{ \frac{x-1}{n} } \rangle \cdot \langle \pi_s, \overrightarrow{i}_{\varepsilon}^{ \frac{x}{n} } \rangle - \overleftarrow{\eta}_s^{\varepsilon n}(x) \overrightarrow{\eta}_s^{\varepsilon n}(x+1) | \lesssim (\varepsilon n)^{-1}$. This together  with \eqref{L1frac}, gives
\begin{align*}
\Big| \frac{1}{2n} \sum_{x} \big[ \langle \pi_s,  \overleftarrow{i}_{\varepsilon}^{ \frac{x-1}{n} } \rangle \cdot \langle \pi_s, \overrightarrow{i}_{\varepsilon}^{ \frac{x}{n} } \rangle - \overleftarrow{\eta}_s^{\varepsilon n}(x) \overrightarrow{\eta}_s^{\varepsilon n}(x+1) \big]  [ -(-\Delta)^{\gamma/2}G_s](\tfrac{x}{n})  \Big| \lesssim (\varepsilon n)^{-1}.
\end{align*}
Then \eqref{char95} is equal to zero if we can prove that
\begin{align*} 
&\limsup_{\varepsilon \rightarrow 0^{+}} \liminf_{n \rightarrow \infty} \,\, \mathbb{P}_{\mu_n}\Big( \sup_{t \in [0,T]} \Big| \int_{0}^{t} \Big\{ \frac{1}{2n} \sum_{x}   n^{\gamma}  \mcb{K}_n G_s (\tfrac{x}{n}) [ \eta_s^n(x-1)  \eta_s^n(x) +  \eta_s^n(x) \eta_s^n(x+1) ] \\
+& \frac{n^{\gamma}}{2n} \mcb {R}_{n}^G(s)  - \frac{1}{2n} \sum_{x} \big[ [ -(-\Delta)^{\gamma/2}G_s](\tfrac{x-1}{n}) -  [ -(-\Delta)^{\gamma/2}G_s](\tfrac{x}{n}) \big]  \overleftarrow{\eta}_s^{\varepsilon n}(x-1) \overrightarrow{\eta}_s^{\varepsilon n}(x) \\
-& \frac{1}{2n} \sum_{x} \big[ \overleftarrow{\eta}_s^{\varepsilon n}(x-1) \overrightarrow{\eta}_s^{\varepsilon n}(x)  [ -(-\Delta)^{\gamma/2}G_s](\tfrac{x}{n}) + \overleftarrow{\eta}_s^{\varepsilon n}(x) \overrightarrow{\eta}_s^{\varepsilon n}(x+1)  [ -(-\Delta)^{\gamma/2}G_s](\tfrac{x}{n}) \big]   \Big\} ds\Big| > \frac{\delta}{32}\Big)=0.
\end{align*}
Last display is bounded from above by the sum of the next three terms
\begin{equation}  \label{char10}
\begin{split}
&\limsup_{\varepsilon \rightarrow 0^{+}} \liminf_{n \rightarrow \infty} \,\, \mathbb{P}_{\mu_n}\Big( \sup_{t \in [0,T]} \Big| \int_{0}^{t} \Big\{  \frac{n^{\gamma}}{2n} \mcb {R}_{n}^G(s)  \\
+& \frac{1}{2n}  \sum_{x} \overleftarrow{\eta}_s^{\varepsilon n}(x-1)  \overrightarrow{\eta}_s^{\varepsilon n}(x) \big( [ -(-\Delta)^{\gamma/2}G_s](\tfrac{x-1}{n}) - [ -(-\Delta)^{\gamma/2}G_s](\tfrac{x}{n}) \big) \Big\} ds \Big| > \frac{\delta}{96}\Big), 
\end{split}
\end{equation}
\begin{align} \label{char11}
\limsup_{\varepsilon \rightarrow 0^{+}} \liminf_{n \rightarrow \infty} \,\, \mathbb{P}_{\mu_n} \Big( \sup_{t \in [0,T]} \Big| \int_{0}^{t} \frac{2}{n} \sum_{x}  \big| n^{\gamma} \mcb{K}_n G_s (\tfrac{x}{n})-    [ -(-\Delta)^{\gamma/2}G_s](\tfrac{x}{n})   \big|  ds \Big| > \frac{\delta}{96} \Big),
\end{align}
\begin{equation} \label{char12}
\begin{split}
\limsup_{\varepsilon \rightarrow 0^{+}} \liminf_{n \rightarrow \infty} \,\, \mathbb{P}_{\mu_n} \Big( \sup_{t \in [0,T]} \Big| \int_{0}^{t} &\frac{1}{n} \sum_{x} \big( [ -(-\Delta)^{\gamma/2}G_s](\tfrac{x}{n}) + [ -(-\Delta)^{\gamma/2}G_s](\tfrac{x+1}{n}) \big)  \\
\times & [ \eta_s^n(x)  \eta_s^n(x+1)  - \overleftarrow{\eta}_s^{\varepsilon n}(x) \overrightarrow{\eta}_s^{\varepsilon n}(x+1)  ] ds\Big| > \frac{\delta}{96} \Big). 
\end{split}
\end{equation}
In \eqref{char11}, we used the fact that $|\eta_s^n (\cdot)| \leq 1$, for every $s\in[0,T]$. Combining the fact that $|\eta_s^n (\cdot)| \leq 1$, for every $s\in[0,T]$ with Corollary \ref{corfrac}, Proposition \ref{convext} and Markov's inequality, \eqref{char10} is equal to zero. From the fact that $|\eta_s^n (\cdot)| \leq 1$, for every $s\in[0,T]$, Proposition \ref{convdisc} and Markov's inequality, \eqref{char11} is equal to zero. Finally, since $[ -(-\Delta)^{\gamma/2}G(s, \cdot)] \in L^1(\mathbb{R}) \cap L^{\infty}(\mathbb{R})$, from Lemma \ref{replacement} and Markov's inequality, \eqref{char12} is equal to zero. This ends the proof.
\end{proof}
Now we prove that any limit point $\mathbb{Q}$ of the sequence $(\mathbb{Q}_n)_{n \geq 1}$ is concentrated on trajectories of measures $\pi_t(du)=\rho(t,u)du$
such that $\rho$ satisfies condition (2) of  Definition \ref{eq:dif}. 
\subsection{The validity of condition (2) in Definition \ref{eq:dif} } \label{sec_4.2}
Similarly to \cite{casodif}, we begin with an important result that does not depend on the dynamics but only on \eqref{defcb}.
\begin{prop} \label{estenergstat1}
It holds
\begin{align*} 
\mathbb{Q} \Big( \pi_{\cdot} \in \mcb{D} ([0,T], \mcb{M}^{+} ): \int_0^T \int_{\mathbb{R}} [ \rho^2(t,u) - b^2]^2   du dt < \infty \Big) = 1.
\end{align*}
\end{prop}
\begin{proof}
From Section 5.1 of \cite{casodif}, we know that
\begin{align*} 
\mathbb{Q} \Big( \pi_{\cdot} \in \mcb{D} ([0,T], \mcb{M}^{+} ): \int_0^T \int_{\mathbb{R}} [ \rho(t,u) - b]^2   du dt < \infty \Big) = 1.
\end{align*}
Since $[ \rho^2(t,u) - b^2]^2 = [ \rho(t,u) + b]^2 [ \rho(t,u) - b]^2 \leq 4 [ \rho(t,u) - b]^2$ for every $(t,u) \in [0,T] \times \mathbb{R}$, we have the desired result.
\end{proof}
The main goal now  is to prove the next result.
\begin{prop}\label{energy estimates}
For $\gamma\in(0,2)$, the measure $\mathbb{Q}$ is concentrated on trajectories of measures $\pi_t(du) =\rho(t,u)du$, such that \begin{equation*}
\mathbb{Q} \Big( \pi_{\cdot} \in \mcb{D} ([0,T], \mcb{M}^{+} ): \int_{0}^{T}  \iint_{\mathbb{R}^2} \frac{[ \rho^2(t,u)  -   \rho^2(t,v) ]^2 }{|u-v|^{1+\gamma}} du dv  dt < \infty \Big) = 1.
\end{equation*}
\end{prop}
We observe that the second condition of weak solution of \eqref{fracpor} is a direct consequence of the two previous results. Before we prove Proposition \ref{energy estimates}, we establish some estimates on the Dirichlet form which are needed in the proof of the previous proposition.
We define the Dirichlet form  by  $\langle \sqrt{f},- \mcb L \sqrt{f} \rangle_{\nu_{b}}$,
where $f: \Omega \rightarrow \mathbb{R}$ is a density with respect to $\nu_{b}$ and for all functions $g,h: \Omega \rightarrow \mathbb{R}$, $\langle g,h \rangle_{\nu_{b}}$ denotes the scalar product in $L^{2} (\Omega, \nu_{b} )$. The quadratic form associated to $\mcb L$ is the operator  $\mcb D$  given by
\begin{equation}\label{D_P}
\mcb D(\sqrt{f},\nu_{b})\;:=\;\frac{1}{4} \sum_{x,y }    p(x-y) I_{x,y}(\sqrt{f},\nu_{b}),
\end{equation}
where $I_{x,y}(\sqrt{f},\nu_{b}):= \int_{\Omega} \tilde{c}_{x,y}(\eta) [\sqrt{f(\eta^{x,y})}-\sqrt{f(\eta)} ]^{2}  d\nu_{b} = \int_{\Omega} {c}_{x,y}(\eta)  [\sqrt{f(\eta^{x,y})}-\sqrt{f(\eta)} ]^{2}  d\nu_{b} $.
 Above $c_{x,y}(\eta)$ and $\tilde{c}_{x,y}(\eta)$ are given, respectively, in \eqref{rates} and \eqref{eq:rates_porous}. From Remark \ref{remsep}, we get 
\begin{equation}\label{D_S}
\mcb D(\sqrt{f},\nu_{b}) \geq \, \frac{1}{4} \sum_{|x-y|=1 }  p(x-y) \int_{\Omega}  [\sqrt{f(\eta^{x,y})}-\sqrt{f(\eta)} ]^{2}  d\nu_{b}:=\mcb D_{NN}(\sqrt{f},\nu_{b}).
\end{equation}  
Observe that by a change of variables it is easy to check that 
\begin{equation}\label{bound-Dirichlet form}
	\langle \mcb L\sqrt{f}, \sqrt{f} \rangle_{\nu_{b}} = -\frac{1}{2}  \mcb D(\sqrt{f},\nu_{b}) \leq - \frac{1}{2}  \mcb D_{NN} (\sqrt{f},\nu_{b}).
	\end{equation}
Now we will prove the main result of this subsection. Recall \eqref{defcb}.
\begin{proof}[Proof of Proposition \ref{energy estimates}]
It is enough to prove that there exists $C_1>0$ independent of $\ve$ such that for every $\ve >0$
\begin{align} \label{claim1}
\mathbb{E}_{\mathbb{Q}} \Big[  \int_0^T  \iint_{Q_{\varepsilon}} \frac{ [\rho^2(t,v) - \rho^2(t,u)]^2 }{|u-v|^{1+\gamma}} du dv  dt \Big]  \leq C_1, 
\end{align}
where $Q_{\varepsilon}:=\{ (u,v) \in \mathbb{R}^2: |u-v| \geq \varepsilon \}$. Indeed, the desired result is a direct consequence of \eqref{claim1} and the Monotone Convergence Theorem. To prove last claim, note that from Riesz's Representation Theorem, it is enough to show that there exist positive constants $C_2, C_3>0$ independent of $\varepsilon$ such that
\begin{align} \label{claim2a}
  \mathbb{E}_{\mathbb{Q}} \Big[  \sup_{F} \Big\{ \int_0^T  \iint_{Q_{\varepsilon}} \Big\{ \frac{ [\rho^2(t,v) - \rho^2(t,u)] F(t,u,v) }{|u-v|^{1+\gamma}} - C_2  \frac{ [F(t,u,v)]^2 }{|u-v|^{1+\gamma}} \Big\} du dv  dt \Big\} \Big]  \leq C_3, 
\end{align}
for every $\varepsilon >0$, where the supremum above is carried over $F \in C_c^{0,2} \big( (0,T) \times \mathbb{R}^2 \big)$; we choose this space of test functions since it is dense in the Hilbert space $L^2 \big( (0,T) \times \mathbb{R}^2, d \mu_{\varepsilon} \big)$, where $\mu_{\varepsilon}$ is the measure whose density, with respect to the Lebesgue measure, is given by $(t,u,v) \in (0,T) \times \mathbb{R}^2 \rightarrow \mathbbm{1}_{\{|u-v| \geq \varepsilon\}} |u-v|^{-1-\gamma}$, for every $\varepsilon >0$. It is enough to prove \eqref{claim2} with the supremum outside the expectation, since we can always use Lemma 7.5 in \cite{supinsexp} to insert this supremum inside  the expectation. Therefore from here on we fix $\varepsilon >0$ and $F$ in $C_c^{0,2} \big( (0,T) \times \mathbb{R}^2 \big)$. Combining the fact that $F \in L^1 \big( (0,T) \times \mathbb{R}^2, d \mu_{\varepsilon} \big)$ with \eqref{aproxid}, we get
\begin{align*}
\limsup_{ \varepsilon_1 \rightarrow 0^+ } \mathbb{E}_{\mathbb{Q}} \Big[   \int_0^T  \iint_{Q_{\varepsilon}}  \Big\{ \frac{ \big( [\rho^2(t,v) - \rho^2(t,u)] - [ \langle \rho_s, \overleftarrow{i_{\varepsilon_1}^{ v}} \rangle \langle \rho_s, \overrightarrow{i_{\varepsilon_1}^{ v}} \rangle - \langle \rho_s, \overleftarrow{i_{\varepsilon_1}^{ u}} \rangle \langle \rho_s, \overrightarrow{i_{\varepsilon_1}^{ u}} \rangle ] \big) F(t,u,v) }{|u-v|^{1+\gamma}}  \Big\} du dv  dt  \Big].
\end{align*}
Therefore we obtain \eqref{claim2a} if we can show that
\begin{align} \label{claim2b}
\limsup_{ \varepsilon_1 \rightarrow 0^+ } \mathbb{E}_{\mathbb{Q}} \Big[   \int_0^T  \iint_{Q_{\varepsilon}}  \Big\{ \frac{  [ \langle \rho_s, \overleftarrow{i_{\varepsilon_1}^{ v}} \rangle \langle \rho_s, \overrightarrow{i_{\varepsilon_1}^{ v}} \rangle - \langle \rho_s, \overleftarrow{i_{\varepsilon_1}^{ u}} \rangle \langle \rho_s, \overrightarrow{i_{\varepsilon_1}^{ u}} \rangle ]  F(t,u,v) }{|u-v|^{1+\gamma}} - C_2  \frac{ [F(t,u,v)]^2 }{|u-v|^{1+\gamma}} \Big\} du dv  dt \Big\} \Big]  \leq C_3, 
\end{align}
where $C_2$ and $C_3$ are positive constants which do not depend on $\varepsilon$ and $G$. Next we observe that the function $\Lambda^F:
\mcb{D}([0,T],  \mcb {M}^{+}) \rightarrow \mathbb{R}$ given by
\begin{align*}
\Lambda^F(\pi) := & \int_0^T  \iint_{Q_{\varepsilon}}  \Big\{ \frac{  [ \langle \pi_s, \overleftarrow{i_{\varepsilon_1}^{ v}} \rangle \langle \pi_s, \overrightarrow{i_{\varepsilon_1}^{ v}} \rangle - \langle \pi_s, \overleftarrow{i_{\varepsilon_1}^{ u}} \rangle \langle \pi_s, \overrightarrow{i_{\varepsilon_1}^{ u}} \rangle ]  F(t,u,v) }{|u-v|^{1+\gamma}} - C_2  \frac{ [F(t,u,v)]^2 }{|u-v|^{1+\gamma}} \Big\} du dv  dt \Big\} \\
= & \int_0^T  \iint_{Q_{\varepsilon}}  \Big\{ \frac{  [ \langle \rho_s, \overleftarrow{i_{\varepsilon_1}^{ v}} \rangle \langle \rho_s, \overrightarrow{i_{\varepsilon_1}^{ v}} \rangle - \langle \rho_s, \overleftarrow{i_{\varepsilon_1}^{ u}} \rangle \langle \rho_s, \overrightarrow{i_{\varepsilon_1}^{ u}} \rangle ]  F(t,u,v) }{|u-v|^{1+\gamma}} - C_2  \frac{ [F(t,u,v)]^2 }{|u-v|^{1+\gamma}} \Big\} du dv  dt \Big\}  
\end{align*}
is lower semi-continuous and bounded with respect to the Skorohod topology of $\mcb{D}([0,T],  \mcb {M}^{+})$. Plugging this with the definition of $\mathbb{Q}_n$ and the fact that $\mathbb{Q}$ is the limit of some subsequence $\mathbb{Q}_{n_j}$, the limit in \eqref{claim2b} is bounded from above by
\begin{equation} \label{claim2}
\begin{split}
&\limsup_{\varepsilon_1 \rightarrow 0^{+}} \limsup_{n \rightarrow \infty}  \mathbb{E}_{\mu_n} \Big[   \int_0^T n^{\gamma-1} \!\!\sum_{x,y:|x-y| \geq \varepsilon n}  \big[ [ \overleftarrow{\eta}_t^{\varepsilon_1 n} (y) \overrightarrow{\eta}_t^{\varepsilon_1 n} (y+1) - \overleftarrow{\eta}_t^{\varepsilon_1 n} (x) \overrightarrow{\eta}_t^{\varepsilon_1 n} (x+1) ] F(t, \tfrac{x}{n}, \tfrac{y}{n}) \\
-& C_2 [F(t, \tfrac{x}{n}, \tfrac{y}{n})]^2 \big] (c_{\gamma} )^{-1} p (x-y)   dt  \Big]
\end{split}
\end{equation} 
Since $F \in C_c^{0,2} \big( (0,T) \times \mathbb{R}^2 \big)$, there exists $b_F>0$ such that $F(t,u,v)=0$ if $|u| \geq b_F$ or $|v| \geq b_F$. Now we define $(\Phi_n)_{n \geq 1}: [0,T] \times \mathbb{R} \rightarrow \mathbb{R}$ by
\begin{align*}
\forall t \in [0,T], \forall x \in \mathbb{Z}, \quad	\Phi_n (t, \tfrac{x}{n} ) := n^{\gamma} \sum_{y: |x-y| \geq \varepsilon n} F(t, \tfrac{x}{n}, \tfrac{y}{n}) p (y-x). 
\end{align*}
We observe that $\Phi_n (t, \tfrac{x}{n} )=0$ if $|x| \geq b_F$ and

	\begin{align*}
		 |\Phi_n (t, \tfrac{x}{n} ) | \leq \| F \|_{\infty} n^{\gamma} \sum_{y: |x-y| \geq \varepsilon n} p (x-y) \lesssim \| F \|_{\infty} \varepsilon^{-\gamma}; \quad  \quad \frac{1}{n} \sum_{x} |\Phi_n (t, \tfrac{x}{n} ) | \lesssim b_F \| F \|_{\infty} \varepsilon^{-\gamma}.
	\end{align*}  
Since $\varepsilon >0$ is fixed, we see that $(\Phi_n)_{n \geq 1}$ satisfies \eqref{boundrep}.  From Lemma \ref{replacement}, we get
\begin{align*}
 \limsup_{\varepsilon_1 \rightarrow 0^{+}} \limsup_{n \rightarrow \infty} \mathbb{E}_{\mu_n} \Big[   \int_0^T n^{\gamma-1} \sum_{x,y:|x-y| \geq \varepsilon n}  & [ \overleftarrow{\eta}_t^{\varepsilon_1 n} (x) \overrightarrow{\eta}_t^{\varepsilon_1 n} (x+1) \\
 -& \eta_t^n(x) \eta_t^n(x+1) ] F(t, \tfrac{x}{n}, \tfrac{y}{n}) (c_{\gamma} )^{-1} p (x-y) dt \Big] =0.
\end{align*}
In an analogous way, we get that
\begin{align*}
 \limsup_{\varepsilon_1 \rightarrow 0^{+}} \limsup_{n \rightarrow \infty} \mathbb{E}_{\mu_n} \Big[   \int_0^T n^{\gamma-1} \sum_{x,y:|x-y| \geq \varepsilon n}  & [ \overleftarrow{\eta}_t^{\varepsilon_1 n} (y) \overrightarrow{\eta}_t^{\varepsilon_1 n} (y+1) \\
 -& \eta_t^n(y) \eta_t^n(y+1) ] F(t, \tfrac{x}{n}, \tfrac{y}{n}) (c_{\gamma} )^{-1} p (x-y) dt \Big] =0.
\end{align*}
Then in order to obtain \eqref{claim2}, it is enough to find $C_2, C_3 > 0$ such that
\begin{align} \label{claim3}
\limsup_{\varepsilon_1 \rightarrow 0^{+}} \limsup_{n \rightarrow \infty} \mathbb{E}_{\mu_n} \Big[   \int_0^T n^{\gamma-1} \sum_{x,y:|x-y| \geq \varepsilon n}  \big\{ & [ \eta_t^{ n} (y) \eta_t^{ n} (y+1) - \eta_t^{ n} (x) \eta_t^{ n} (x+1) ] F(t, \tfrac{x}{n}, \tfrac{y}{n}) \nonumber  \\
- & C_2 [F(t, \tfrac{x}{n}, \tfrac{y}{n})]^2 \big\}p (x-y)   dt  \Big] \leq C_3 c_{\gamma}.
\end{align}
Observe that
\begin{align}
&  \int_0^T n^{\gamma-1} \sum_{x,y:|x-y| \geq \varepsilon n} [ \eta_t^{ n} (y) \eta_t^{ n} (y+1) - \eta_t^{ n} (x) \eta_t^{ n} (x+1) ] F(t, \tfrac{x}{n}, \tfrac{y}{n})p (x-y) dt \nonumber \\
=&  \int_0^T  \frac{n^{\gamma-1}}{2} \sum_{x,y:|x-y| \geq \varepsilon n} [\eta_t^{ n} (y) - \eta_t^{ n} (x)] \tilde{c}_{x,y}( \eta_t^{n} ) F(t, \tfrac{x}{n}, \tfrac{y}{n}) p (x-y) dt \nonumber \\
+&  \int_0^T  \frac{n^{\gamma-1}}{2}\sum_{x,y:|x-y| \geq \varepsilon n}  [ \eta_t^{ n} (y-1) \eta_t^{ n} (y) - \eta_t^{ n} (x-1) \eta_t^{ n} (x) ] [F(t, \tfrac{x-1}{n}, \tfrac{y-1}{n}) - F(t, \tfrac{x}{n}, \tfrac{y}{n}) ]  p (x-y) dt \nonumber\\
+ &  \int_0^T  \frac{n^{\gamma-1}}{2} \sum_{x,y:|x-y| \geq \varepsilon n} \eta_t^{ n} (x) \eta_t^{ n} (y+1) [F(t, \tfrac{x}{n}, \tfrac{y}{n}) - F(t, \tfrac{x+1}{n}, \tfrac{y+1}{n}) ]  p (x-y) dt \nonumber \\
+&  \int_0^T  \frac{n^{\gamma-1}}{2} \sum_{x,y:|x-y| \geq \varepsilon n}\eta_t^{ n} (x) \eta_t^{ n} (y-1) [F(t, \tfrac{x}{n}, \tfrac{y}{n}) - F(t, \tfrac{x-1}{n}, \tfrac{y-1}{n}) ]  p (x-y) dt. \nonumber
\end{align}
Since $\varepsilon$ is fixed, performing some Taylor expansions on $F$, we conclude that the last three terms in last display vanish as $n$ goes to infinity and only the first one remains. Then we need to study
\begin{align*}
\mathbb{E}_{\mu_n} \Big[   \int_0^T  \frac{n^{\gamma-1}}{2} \sum_{x,y:|x-y| \geq \varepsilon n} [\eta_t^{ n} (y) - \eta_t^{ n} (x)] \tilde{c}_{x,y}( \eta_t^{n} ) F(t, \tfrac{x}{n}, \tfrac{y}{n}) p (x-y) dt  \Big].
\end{align*}
 From entropy inequality plus \eqref{defcb},  Jensen's inequality and Feyman-Kac formula ( see \cite{kipnis1998scaling}), last display is bounded from above by 
\begin{align} 
C_b + \int_0^T \sup_{f}   \Big\{  \frac{n^{\gamma-1}}{2} \sum_{x,y:|x-y| \geq \varepsilon n}  F \left( t, \tfrac{x}{n},\tfrac{y}{n}\right)p(y-x)& \int_{\Omega} [\eta(y)-\eta(x) ] \tilde{c}_{x,y}(\eta)f(\eta)\,d\nu_{b} \nonumber \\
+ &n^{\gamma-1}\langle \mcb L\sqrt{f},\sqrt{f} \rangle_{\nu_{b}} \Big\} dt, \label{est2} 
\end{align}
where the supremum above is carried over all the densities $f$ with respect to $\nu_{b}$. In  last inequality we also used the facts that $e^{|u|} \leq e^{u}+e^{-u}$ and 
\begin{align} \label{ineqlimsup}
\limsup_{n \rightarrow \infty} \log(a_n + b_n ) = \max \Big\{ \limsup_{n \rightarrow \infty} \log(a_n  ) , \limsup_{n \rightarrow \infty} \log( b_n )   \Big\}.
\end{align}
From the change of variables $\eta$ to $\eta^{x,y}$ and Remark \ref{revmeas}, we can rewrite the first term inside the supremum in  \eqref{est2}  as
\begin{align} 
&\;  \frac{n^{\gamma-1}}{2} \sum_{x,y:|x-y| \geq \varepsilon n} F \left( t, \tfrac{x}{n},\tfrac{y}{n}\right)p(y-x)\int_{\Omega} \eta(y) \tilde{c}_{x,y}(\eta) [ f(\eta) - f(\eta^{x,y})] d\nu_{b}. \label{est3}  
\end{align}
Note that $f(\eta)-f(\eta^{x,y}) = [\sqrt{f(\eta)}-\sqrt{f(\eta^{x,y})} ] [ \sqrt{f(\eta)}+\sqrt{f(\eta^{x,y})} ]$. From Young's inequality, for any $A>0$ we can bound \eqref{est3} from above by
\begin{align*}
&\frac{n^{\gamma-1}A}{4} \sum_{x,y:|x-y| \geq \varepsilon n}  [F (t, \tfrac{x}{n},\tfrac{y}{n} ) ]^{2} p(y-x)\int_{\Omega} \eta(y) \tilde{c}_{x,y}(\eta) [\sqrt{f(\eta)}+\sqrt{f(\eta^{x,y})} ]^{2} d\nu_{b} \\
+& \frac{n^{\gamma-1}}{4A} \sum_{x,y:|x-y| \geq \varepsilon n} p(y-x) \int_{\Omega} \eta(y) \tilde{c}_{x,y}(\eta) [\sqrt{f(\eta)}-\sqrt{f(\eta^{x,y})} ]^{2}  d\nu_{b}.
\end{align*}
Since $|\tilde{c}_{x,y}(\eta)| \leq 4$, $|\eta(y)|\leq 1$ and $f$ is a density with respect to $\nu_{b}$, the previous expression can be bounded from above by
\begin{align*}
& 4n^{\gamma-1}A\sum_{x,y:|x-y| \geq \varepsilon n}  [F (t, \tfrac{x}{n},\tfrac{y}{n} ) ]^{2}p(y-x)+ \frac{n^{\gamma-1}}{4A} \sum_{x,y:|x-y|\geq \varepsilon n}p(y-x)I_{x,y}(\sqrt f, \nu_b)\\
\leq & 4n^{\gamma-1}A \sum_{x,y:|x-y| \geq \varepsilon n} [F (t, \tfrac{x}{n},\tfrac{y}{n} ) ]^{2}p(y-x) + \frac{n^{\gamma-1}}{A} \mcb D (\sqrt{f},\nu_{b} ).
\end{align*}
Therefore, choosing $A=2$, from  \eqref{bound-Dirichlet form} the expression inside the supremum in \eqref{est2} is bounded from above by
\begin{equation}\label{exp9}
8n^{\gamma-1} \sum_{x,y:|x-y| \geq \varepsilon n} [F (t, \tfrac{x}{n},\tfrac{y}{n} ) ]^{2}p(y-x) 
\end{equation} 
In particular, \eqref{claim3} holds with $C_2=8$ and $C_3 = \frac{C_b}{c_{\gamma}}$, leading to the desired result.
\end{proof}
\section{Replacement lemmas} \label{secreplem}
The goal of this section is to prove Lemma \ref{replacement}, which is useful in order to produce the main results of Section \ref{secchar}.
 \begin{lem} \textbf{(Replacement Lemma)} \label{replacement}
Assume $(\Phi_{n})_{n \geq 1} : [0,T] \times \mathbb{R}  \rightarrow \mathbb{R}$ satisfies
	\begin{align} \label{boundrep}
		\frac{1}{n} \sum_{x} \sup_{s \in [0,T]} | \Phi_{n}(s, \tfrac{x}{n} ) | \leq M_1 \quad\textrm{and}\quad\| \Phi_{n} \|_{\infty}:= \sup_{(s,u) \in [0,T] \times \mathbb{R}} | \Phi_n (s,u)| \leq M_2,
	\end{align}
for every $n \geq 1$. Then for every $t \in [0,T]$, it holds 
\begin{equation*}% \label{rl1}
		\limsup_{\varepsilon \rightarrow 0^{+}}\limsup_{n \rightarrow \infty}\mathbb{E}_{\mu_n} \Big[ \Big| \int_{0}^{t}\frac{1}{n}\sum_{x }  \Phi_{n} (s, \tfrac{x}{n} ) [\eta_{s}(x)\eta_{s}(x+1)- \overleftarrow{\eta}_{s}^{\varepsilon n}(x)\overrightarrow{\eta}_{s}^{\varepsilon n}(x+1) ]ds \Big| \Big] =0.
	\end{equation*}
\end{lem}

Now we describe the strategy of the proof, which is accomplished in three steps that we now describe. 
In the first step,  Lemma \ref{lemrep1}, we replace $\eta(x)\eta(x+1)$ by $\overleftarrow{\eta}^{\ell}(x)\eta(x+1)$, with $\ell = \varepsilon n^{\gamma/2}$. In the second step, Lemma \ref{lemrep2},
	we  replace $\overleftarrow{\eta}^{\ell}(x)\eta(x+1)$ by $\overleftarrow{\eta}^{\ell}(x)\overrightarrow{\eta}^{\varepsilon n}(x+1)$, for  $\ell = \varepsilon n^{\frac{\gamma}{2}}$. In the  third step, Lemma \ref{lemrep3}, we  replace $\overleftarrow{\eta}^{\ell}(x)\overrightarrow{\eta}^{\varepsilon n}(x+1)$ by $\overleftarrow{\eta}^{\varepsilon n}(x)\overrightarrow{\eta}^{\varepsilon n}(x+1)$, for $\ell = \varepsilon n^{\frac{\gamma}{2}}$. 
Following the procedure described above, Lemma \ref{replacement} is a direct consequence of last results.  We begin with the first step. We do not present its proof, since it is very similar to the proof of Lemma 5.3 in \cite{MR4099999}, but now, one has to take into account the fact that summations are running over $\mathbb Z$ and that nearest-neighbor jumps are always possible, see Remark \ref{remsep}.
\begin{lem} \label{lemrep1}
Assume $(\Phi_{n})_{n \geq 1}: [0,T] \times \mathbb{R} \rightarrow \mathbb{R}$ satisfies \eqref{boundrep} and denote $\ell = \varepsilon n^{\frac{\gamma}{2}}$. Then for every $t \in [0,T]$, it holds
	\begin{equation}  \label{rl1rl}
		\limsup_{\varepsilon \rightarrow 0^{+}}\limsup_{n \rightarrow \infty}\mathbb{E}_{\mu_n} \Big[ \Big| \int_{0}^{t}\frac{1}{n}\sum_{x} \Phi_{n} (s, \tfrac{x}{n} ) [\eta_{s}(x)- \overleftarrow{\eta}_{s}^{\ell}(x) ]\eta_{s}(x+1) ds \Big| \Big] =0.
	\end{equation}
\end{lem}
Now we state an auxiliary lemma, which is an alternative version of Lemma 5.8 in \cite{stefano}. It will be useful in the proof of Lemma \ref{lemrep2} and Lemma \ref{lemrep3}. 
\begin{lem} \textbf{(Moving Particle Lemma)} Fix $r \neq 0 \in \mathbb{Z}$ and $f$ a density with respect to $\nu_{b}$ on $\Omega$. For every $x \in \mathbb{Z}$, let  $\Omega_x:=\{ \eta \in \Omega : \eta(x-2)=\eta(x-1)=1 \}$. Then
\begin{align*}
\sum_{x} \int_{\Omega_x} \big[\sqrt{f \left( \eta^{x,x+r} \right) } - \sqrt{f \left( \eta \right) } \big]^2 d \nu_{b} \lesssim  |r|^{\gamma} \mcb D (\sqrt{f}, \nu_{b} ) .
\end{align*}
\end{lem}
\begin{proof}
Without loss of generality, we can assume that $r > 0$. First fix $x \in \mathbb{Z}$. Our goal is to exchange  particles at the bond $\{x,x+r\}$. We will do so through a variety of paths, defined using an intermediate site between $x$ and $x+r$. There are two possibilities: there exists $m \in \mathbb{N}$ such that $r=2 m$ (case 1) or there exists $m \in \mathbb{N}$ such that $r=2 m-1$ (case 2). In the first case, for every $j \in \{1, \ldots, m\}$, we choose $z_{1j}:=x+m+j$. In the later one, for every $j \in \{1, \ldots, m\}$ we choose $z_{1j}:=x+m-1+j$. Moreover, for every $j$, denote $z_{0j}:=x$ and $z_{2j}:=x+r$. We observe that there are no repetitions inside the set $\{|z_{1j}-z_{0j}|, |z_{2j}-z_{1j}|; j=1, \ldots, m\}$. Now fix $j \in \{1, \ldots, m\}$. Moreover, $\sqrt{f \left( \eta^{x,x+r} \right) } - \sqrt{f \left( \eta \right) }=0$ when $\eta(x)=\eta(x+r)$, therefore we can restrict ourselves to the configurations $\eta \in \Omega_x$ such that $\eta(x) \neq \eta(x+r)$. First we focus on the following subset of $\Omega_x$:

\begin{align*}
 \Omega_{j,x}^1 := \{ \eta \in \Omega_x:\; \eta(x)=1,\eta(z_{1j})=0,\eta(x+r)=0\}.% ;\\
\end{align*}
Now we will illustrate the sequence of operations performed in order to go from the initial configuration $\eta_0:=\eta$ to $\eta_6:=\eta^{x,x+r}$  for the set $ \Omega_{j,x}^1$. 

\begin{figure}[H]  \begin{center}
		\begin{tikzpicture}[scale=0.53]
			%------------------------------------------------------
			%configuração inicial
			\draw [line width=1] (-7,10.5) -- (13,10.5) ; %semi-reta com seta <-
			\foreach \x in  {-7,-6,-5,-4,-3,-2,-1,0,1,2,3,4,5,6,7,8,9,10,11,12,13}
			\draw[shift={(\x,10.5)},color=black, opacity=1] (0pt,4pt) -- (0pt,-4pt) node[below] {};
			\draw[] (-2.8,10.5) node[] {};
			
			%colocando os indices
				\draw[] (-7,10.5) node[above] {\footnotesize{$\eta_0$}};
				\draw[] (-3,10.3) node[below] {\footnotesize{$x$}};
			\draw[] (4,10.3)  node[below] {\footnotesize{$z_{1j}$}};
			\draw[] (12,10.3) node[below] {\footnotesize{$x+ r$}};
			
			%partículas da conf inicial
			\shade[shading=ball, ball color=black!50!] (-5,10.65) circle (.3);
			\shade[shading=ball, ball color=black!50!] (-4,10.65) circle (.3);
			\shade[shading=ball, ball color=black!50!] (-3,10.65) circle (.3);
		\end{tikzpicture} 
	\end{center}	
	\begin{center}
		\begin{tikzpicture}[scale=0.53]
			%------------------------------------------------------
			%configuração inicial
			\draw [line width=1] (-7,10.5) -- (13,10.5) ; %semi-reta com seta <-
			\foreach \x in  {-7,-6,-5,-4,-3,-2,-1,0,1,2,3,4,5,6,7,8,9,10,11,12,13}
			\draw[shift={(\x,10.5)},color=black, opacity=1] (0pt,4pt) -- (0pt,-4pt) node[below] {};
			\draw[] (-2.8,10.5) node[] {};
			
			%colocando os indices
				\draw[] (-7,10.5) node[above] {\footnotesize{$\eta_1$}};
				\draw[] (-3,10.3) node[below] {\footnotesize{$x$}};
			\draw[] (4,10.3)  node[below] {\footnotesize{$z_{1j}$}};
			\draw[] (12,10.3) node[below] {\footnotesize{$x+ r$}};
			
			%partículas da conf inicial
			\shade[shading=ball, ball color=black!50!] (-5,10.65) circle (.3);
			\shade[shading=ball, ball color=black!50!] (3,10.65) circle (.3);
			\shade[shading=ball, ball color=black!50!] (-3,10.65) circle (.3);
		\end{tikzpicture} 
	\end{center}
	\begin{center}
		\begin{tikzpicture}[scale=0.53]
			%------------------------------------------------------
			%configuração inicial
			\draw [line width=1] (-7,10.5) -- (13,10.5) ; %semi-reta com seta <-
			\foreach \x in  {-7,-6,-5,-4,-3,-2,-1,0,1,2,3,4,5,6,7,8,9,10,11,12,13}
			\draw[shift={(\x,10.5)},color=black, opacity=1] (0pt,4pt) -- (0pt,-4pt) node[below] {};
			\draw[] (-2.8,10.5) node[] {};
			
			%colocando os indices
				\draw[] (-7,10.5) node[above] {\footnotesize{$\eta_2$}};
				\draw[] (-3,10.3) node[below] {\footnotesize{$x$}};
			\draw[] (4,10.3)  node[below] {\footnotesize{$z_{1j}$}};
			\draw[] (12,10.3) node[below] {\footnotesize{$x+ r$}};
			
			%partículas da conf inicial
			\shade[shading=ball, ball color=black!50!] (2,10.65) circle (.3);
			\shade[shading=ball, ball color=black!50!] (3,10.65) circle (.3);
			\shade[shading=ball, ball color=black!50!] (-3,10.65) circle (.3);
		\end{tikzpicture} 
	\end{center}
\end{figure}
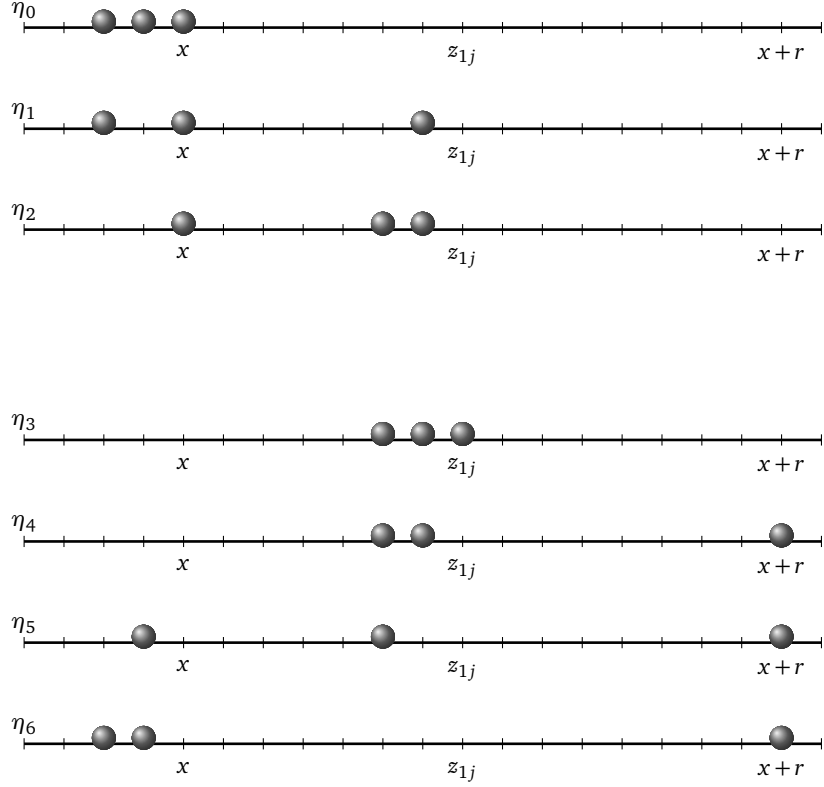
\begin{figure}[H]
\begin{center}
		\begin{tikzpicture}[scale=0.53]
			%------------------------------------------------------
			%configuração inicial
			\draw [line width=1] (-7,10.5) -- (13,10.5) ; %semi-reta com seta <-
			\foreach \x in  {-7,-6,-5,-4,-3,-2,-1,0,1,2,3,4,5,6,7,8,9,10,11,12,13}
			\draw[shift={(\x,10.5)},color=black, opacity=1] (0pt,4pt) -- (0pt,-4pt) node[below] {};
			\draw[] (-2.8,10.5) node[] {};
			
			%colocando os indices
				\draw[] (-7,10.5) node[above] {\footnotesize{$\eta_3$}};
				\draw[] (-3,10.3) node[below] {\footnotesize{$x$}};
			\draw[] (4,10.3)  node[below] {\footnotesize{$z_{1j}$}};
			\draw[] (12,10.3) node[below] {\footnotesize{$x+ r$}};
			
			%partículas da conf inicial
			\shade[shading=ball, ball color=black!50!] (2,10.65) circle (.3);
			\shade[shading=ball, ball color=black!50!] (3,10.65) circle (.3);
			\shade[shading=ball, ball color=black!50!] (4,10.65) circle (.3);
		\end{tikzpicture} 
	\end{center}

\begin{center}
		\begin{tikzpicture}[scale=0.53]
			%------------------------------------------------------
			%configuração inicial
			\draw [line width=1] (-7,10.5) -- (13,10.5) ; %semi-reta com seta <-
			\foreach \x in  {-7,-6,-5,-4,-3,-2,-1,0,1,2,3,4,5,6,7,8,9,10,11,12,13}
			\draw[shift={(\x,10.5)},color=black, opacity=1] (0pt,4pt) -- (0pt,-4pt) node[below] {};
			\draw[] (-2.8,10.5) node[] {};
			
			%colocando os indices
				\draw[] (-7,10.5) node[above] {\footnotesize{$\eta_4$}};
				\draw[] (-3,10.3) node[below] {\footnotesize{$x$}};
			\draw[] (4,10.3)  node[below] {\footnotesize{$z_{1j}$}};
			\draw[] (12,10.3) node[below] {\footnotesize{$x+ r$}};
			
			%partículas da conf inicial
			\shade[shading=ball, ball color=black!50!] (2,10.65) circle (.3);
			\shade[shading=ball, ball color=black!50!] (3,10.65) circle (.3);
			\shade[shading=ball, ball color=black!50!] (12,10.65) circle (.3);
		\end{tikzpicture} 
	\end{center}

\begin{center}
		\begin{tikzpicture}[scale=0.53]
			%------------------------------------------------------
			%configuração inicial
			\draw [line width=1] (-7,10.5) -- (13,10.5) ; %semi-reta com seta <-
			\foreach \x in  {-7,-6,-5,-4,-3,-2,-1,0,1,2,3,4,5,6,7,8,9,10,11,12,13}
			\draw[shift={(\x,10.5)},color=black, opacity=1] (0pt,4pt) -- (0pt,-4pt) node[below] {};
			\draw[] (-2.8,10.5) node[] {};
			
			%colocando os indices
				\draw[] (-7,10.5) node[above] {\footnotesize{$\eta_5$}};
				\draw[] (-3,10.3) node[below] {\footnotesize{$x$}};
			\draw[] (4,10.3)  node[below] {\footnotesize{$z_{1j}$}};
			\draw[] (12,10.3) node[below] {\footnotesize{$x+ r$}};
			
			%partículas da conf inicial
			\shade[shading=ball, ball color=black!50!] (2,10.65) circle (.3);
			\shade[shading=ball, ball color=black!50!] (-4,10.65) circle (.3);
			\shade[shading=ball, ball color=black!50!] (12,10.65) circle (.3);
		\end{tikzpicture} 
	\end{center}

\begin{center}
		\begin{tikzpicture}[scale=0.53]
			%------------------------------------------------------
			%configuração inicial
			\draw [line width=1] (-7,10.5) -- (13,10.5) ; %semi-reta com seta <-
			\foreach \x in  {-7,-6,-5,-4,-3,-2,-1,0,1,2,3,4,5,6,7,8,9,10,11,12,13}
			\draw[shift={(\x,10.5)},color=black, opacity=1] (0pt,4pt) -- (0pt,-4pt) node[below] {};
			\draw[] (-2.8,10.5) node[] {};
			
			%colocando os indices
				\draw[] (-7,10.5) node[above] {\footnotesize{$\eta_6$}};
				\draw[] (-3,10.3) node[below] {\footnotesize{$x$}};
			\draw[] (4,10.3)  node[below] {\footnotesize{$z_{1j}$}};
			\draw[] (12,10.3) node[below] {\footnotesize{$x+ r$}};
			
			%partículas da conf inicial
			\shade[shading=ball, ball color=black!50!] (-5,10.65) circle (.3);
			\shade[shading=ball, ball color=black!50!] (-4,10.65) circle (.3);
			\shade[shading=ball, ball color=black!50!] (12,10.65) circle (.3);
		\end{tikzpicture} 
	\end{center}
	\caption{Operations performed to exchange the particles in the bond $\{x,x+r\}$ for configurations in $\Omega_{j,x}^1$}
	\label{figure-pathcase1}
\end{figure}
Observe that in Figure \ref{figure-pathcase1}, $\tilde{c}_{z_{0j}-1,z_{1j}-1}(\eta_k) \geq 1$ for $k=0,4$, $\tilde{c}_{z_{0j}-2,z_{1j}-2}(\eta_k) \geq 1$ for $k=1,5$. Moreover, $\tilde{c}_{z_{0j},z_{1j}}(\eta_2) \geq 1$ and $\tilde{c}_{z_{1j},z_{2j}}(\eta_3) \geq 1$. This leads to 
\begin{align}
&    \sum_{k=0}^{5} \int_{\Omega_{j,x}^1} [\sqrt{f \left( \eta_{k+1}  \right) } - \sqrt{f \left( \eta_{k} \right) }]^2 d \nu_{b} \label{case1}  \\
\leq &  \sum_{k=0,4} \int_{\Omega_{j,x}^1} \tilde{c}_{z_{0j}-1,z_{1j}-1}(\eta_k) [\sqrt{f \left( \eta_{k}^{z_{0j}-1,z_{1j}-1}  \right) } - \sqrt{f \left( \eta_{k} \right) }]^2 d \nu_{b} \nonumber  \\
+& \sum_{k=1,5} \int_{\Omega_{j,x}^1} \tilde{c}_{z_{0j}-2,z_{1j}-2}(\eta_k) [\sqrt{f \left( \eta_{k}^{z_{0j}-2,z_{1j}-2}  \right) } - \sqrt{f \left( \eta_{k} \right) }]^2 d \nu_{b} \nonumber  \\ 
+ &  \int_{\Omega_{j,x}^1} \tilde{c}_{z_{0j},z_{1j}}(\eta_2) [\sqrt{f \left( \eta_{2}^{ z_{0j},z_{1j} }  \right) } - \sqrt{f \left( \eta_{2} \right) }]^2 d \nu_{b} + \int_{\Omega_{j,x}^1} \tilde{c}_{z_{1j},z_{2j}}(\eta_3) [\sqrt{f \left( \eta_{3}^{ z_{1j},z_{2j} }  \right) } - \sqrt{f \left( \eta_{3} \right) }]^2 d \nu_{b}. \nonumber
\end{align}
For every $k>0$ we perform the transformation $\eta_k \rightarrow \eta_0$ and observe that that the measure $\nu_{b}$ is invariant under these transformations. Then we bound \eqref{case1} from above by
\begin{align}
& 2 \int_{\Omega_{j,x}^1} \tilde{c}_{z_{0j}-1,z_{1j}-1}(\eta_0) [\sqrt{f \left( \eta_{0}^{z_{0j}-1,z_{1j}-1}  \right) } - \sqrt{f \left( \eta_{0} \right) }]^2 d \nu_{b} \label{case1a} \\
+& 2 \int_{\Omega_{j,x}^1} \tilde{c}_{z_{0j}-2,z_{1j}-2}(\eta_0) [\sqrt{f \left( \eta_{0}^{z_{0j}-2,z_{1j}-2}  \right) } - \sqrt{f \left( \eta_{0} \right) }]^2 d \nu_{b} \label{case1b} \\ 
+ &  \int_{\Omega_{j,x}^1} \tilde{c}_{z_{0j},z_{1j}}(\eta_0) [\sqrt{f \left( \eta_{0}^{ z_{0j},z_{1j} }  \right) } - \sqrt{f \left( \eta_{0} \right) }]^2 d \nu_{b} + \int_{\Omega_{j,x}^1} \tilde{c}_{z_{1j},z_{2j}}(\eta_0) [\sqrt{f \left( \eta_{0}^{ z_{1j},z_{2j} }  \right) } - \sqrt{f \left( \eta_{0} \right) }]^2 d \nu_{b}. \label{case1c}
\end{align}
Above, we performed six exchanges of particles in order to go from $\eta$ to $\eta^{x,x+r}$; four auxiliary ones (dealing with the two particles originally at $x-1$ and $x-2$) and two principal ones. First we go from $\eta$ to $\eta_1:=\eta^{z_{0j}-1,z_{1j}-1}$ and then from $\eta_1$ to $\eta_2:=\eta^{z_{0j}-2,z_{1j}-2}$. Now we perform the two principal exchanges, by going from $\eta_2$ to $\eta_3$ and from $\eta_3$ to $\eta_4$. Finally, we return the two auxiliary particles to their original positions, going from $\eta_4$ to $\eta_5:=\eta_4^{z_{0j}-1,z_{1j}-1}$ and from $\eta_5$ to $\eta_6:=\eta^{z_{0j}-2,z_{1j}-2}$. Observe that $\eta_6=\eta^{x,x+r}$ as desired. Since $\eta_0=\eta$, we can write
\begin{align*}
\sqrt{f \left( \eta^{x,x+r}  \right) } - \sqrt{f \left( \eta \right) } = \sum_{k=1}^{6} [\sqrt{f \left( \eta_{k}  \right) } - \sqrt{f \left( \eta_{k-1} \right) }].
\end{align*}
Now we denote
\begin{align*}
&\Omega_{j,x}^2 := \{ \eta \in \Omega_x:\; \eta(x)=0,\eta(z_{1j})=1,\eta(x+r)=1\};
\\
& \Omega_{j,x}^3 := \{ \eta \in \Omega_x:\; \eta(x)=0,\eta(z_{1j})=0,\eta(x+r)=1\};
\\
 &\Omega_{j,x}^4 := \{ \eta \in \Omega_x:\; \eta(x)=1,\eta(z_{1j})=1,\eta(x+r)=0\}.
\end{align*}
From the Cauchy-Schwarz inequality, we have
\begin{align*}
& \int_{\Omega_x} \big[\sqrt{f \left( \eta^{x,x+r} \right) } - \sqrt{f \left( \eta \right) } \big]^2 d \nu_{b}  = \sum_{i=1}^4  \int_{\Omega_{j,x}^i} \big[\sqrt{f \left( \eta^{x,x+r} \right) } - \sqrt{f \left( \eta \right) } \big]^2 d \nu_{b}  \\
\lesssim &  \sum_{i=1}^4  \sum_{k=0}^{5} \int_{\Omega_{j,x}^i} [\sqrt{f \left( \eta_{k+1}  \right) } - \sqrt{f \left( \eta_{k} \right) }]^2 d \nu_{b}. 
\end{align*}
Then from \eqref{case1a}, \eqref{case1b}, \eqref{case1c} and analogous expressions in the remaining sets, we bound last display from above by 
\begin{align}
 & 2 \sum_{i=1}^4 \int_{\Omega_{j,x}^i} \tilde{c}_{z_{0j}-1,z_{1j}-1}(\eta_0) [\sqrt{f \left( \eta_{0}^{z_{0j}-1,z_{1j}-1}  \right) } - \sqrt{f \left( \eta_{0} \right) }]^2 d \nu_{b} \nonumber \\
+ & 2 \sum_{i=1}^4 \int_{\Omega_{j,x}^i} \tilde{c}_{z_{0j}-2,z_{1j}-2}(\eta_0) [\sqrt{f \left( \eta_{0}^{z_{0j}-2,z_{1j}-2}  \right) } - \sqrt{f \left( \eta_{0} \right) }]^2 d \nu_{b}  \nonumber \\
+ & \sum_{i=1}^4 \Big[  \int_{\Omega_{j,x}^i} \tilde{c}_{z_{0j},z_{1j}}(\eta_0) [\sqrt{f \left( \eta_{0}^{ z_{0j},z_{1j} }  \right) } - \sqrt{f \left( \eta_{0} \right) }]^2 d \nu_{b} + \int_{\Omega_{j,x}^i} \tilde{c}_{z_{1j},z_{2j}}(\eta_0) [\sqrt{f \left( \eta_{0}^{ z_{1j},z_{2j} }  \right) } - \sqrt{f \left( \eta_{0} \right) }]^2 d \nu_{b} \Big] \nonumber \\
%= & 2  \int_{\Omega_{x}} \tilde{c}_{z_{0j}-1,z_{1j}-1}(\eta_0) [\sqrt{f \left( \eta_{0}^{z_{0j}-1,z_{1j}-1}  \right) } - \sqrt{f \left( \eta_{0} \right) }]^2 d \nu_{b} \nonumber \\
%+ & 2  \int_{\Omega_{x}} \tilde{c}_{z_{0j}-2,z_{1j}-2}(\eta_0) [\sqrt{f \left( \eta_{0}^{z_{0j}-2,z_{1j}-2}  \right) } - \sqrt{f \left( \eta_{0} \right) }]^2 d \nu_{b} \nonumber \\
%+ &   \int_{\Omega_{x}} \tilde{c}_{z_{0j},z_{1j}}(\eta_0) [\sqrt{f \left( \eta_{0}^{ z_{0j},z_{1j} }  \right) } - \sqrt{f \left( \eta_{0} \right) }]^2 d \nu_{b} + \int_{\Omega_{x}} \tilde{c}_{z_{1j},z_{2j}}(\eta_0) [\sqrt{f \left( \eta_{0}^{ z_{1j},z_{2j} }  \right) } - \sqrt{f \left( \eta_{0} \right) }]^2 d \nu_{b} \Big] \nonumber \\
\leq & 2  \int_{\Omega} \tilde{c}_{z_{0j}-1,z_{1j}-1}(\eta_0) [\sqrt{f \left( \eta_{0}^{z_{0j}-1,z_{1j}-1}  \right) } - \sqrt{f \left( \eta_{0} \right) }]^2 d \nu_{b} \nonumber \\
+ & 2  \int_{\Omega} \tilde{c}_{z_{0j}-2,z_{1j}-2}(\eta_0) [\sqrt{f \left( \eta_{0}^{z_{0j}-2,z_{1j}-2}  \right) } - \sqrt{f \left( \eta_{0} \right) }]^2 d \nu_{b} \nonumber \\
+ &   \int_{\Omega} \tilde{c}_{z_{0j},z_{1j}}(\eta_0) [\sqrt{f \left( \eta_{0}^{ z_{0j},z_{1j} }  \right) } - \sqrt{f \left( \eta_{0} \right) }]^2 d \nu_{b} + \int_{\Omega} \tilde{c}_{z_{1j},z_{2j}}(\eta_0) [\sqrt{f \left( \eta_{0}^{ z_{1j},z_{2j} }  \right) } - \sqrt{f \left( \eta_{0} \right) }]^2 d \nu_{b} \Big] \nonumber \\
\lesssim &  I_{z_{0j}-1,z_{1j}-1}( \sqrt{f}, \nu_b ) +  I_{z_{0j}-2,z_{1j}-2}( \sqrt{f}, \nu_b ) +  I_{z_{0j},z_{1j}}( \sqrt{f}, \nu_b ) +  I_{z_{1j},z_{2j}}( \sqrt{f}, \nu_b ). \label{bound0mpl}
\end{align}
Observe that every jump has length at most $2m$. Since $[p (z_i - z_{i-1})]^{-1} \leq (2m)^{1+\gamma} \leq  8 m^{1+\gamma}$,  we bound the expression in \eqref{bound0mpl} from above by a constant times
\begin{align*}
m^{1+\gamma}  \Big[ &p(z_{1j} - z_{0j} ) I_{z_{0j}-1,z_{1j}-1}( \sqrt{f}, \nu_b ) +p(z_{1j} - z_{0j} )  I_{z_{0j}-2,z_{1j}-2}( \sqrt{f}, \nu_b ) \\
+&p(z_{1j} - z_{0j} )  I_{z_{0j},z_{1j}}( \sqrt{f}, \nu_b ) +p(z_{2j} - z_{1j} )  I_{z_{1j},z_{2j}}( \sqrt{f}, \nu_b ) \Big],
\end{align*}
and this holds for every $j\in\{1,\cdots, m\}$. 
Summing over $j \in \{1, \ldots, m\}$, diving then both sides by $m$ and then summing over  $x\in\mathbb Z$,  we get
\begin{align*}
\sum_{x } & \int_{\Omega_x} \big[\sqrt{f \left( \eta^{x,x+r} \right) } - \sqrt{f \left( \eta \right) } \big]^2 d \nu_{b} \\
\lesssim & m^{\gamma} \Big[  \sum_{x } \sum_{j=1}^{m} p(z_{1j} - z_{0j} ) I_{z_{0j}-1,z_{1j}-1}( \sqrt{f}, \nu_b ) +  \sum_{x } \sum_{j=1}^{m} p(z_{1j} - z_{0j} ) I_{z_{0j}-2,z_{1j}-2}( \sqrt{f}, \nu_b ) \\
+&  \sum_{x } \sum_{j=1}^{m} p(z_{1j} - z_{0j} ) I_{z_{0j},z_{1j}}( \sqrt{f}, \nu_b ) + \sum_{x } \sum_{j=1}^{m} p(z_{2j} - z_{1j} )  I_{z_{1j},z_{2j}}( \sqrt{f}, \nu_b ) \Big].
\end{align*}
We observe the the first three double sums above are equivalent and correspond to jumps with $m$ different lengths ($m+j$ if $r=2m$ or $m-1+j$ if $r=2m-1$), therefore in each one of the first three double sums, every bond  is repeated at most once. On the other hand, the last double sum corresponds to jumps with size $m-j$ (therefore there is no coincidence with the jumps in the first three double sums) and every bond  is repeated at most once. Hence when we combine the four double sums every bond in $\mcb B$ is repeated at most twice, which gives
\begin{align*}
\sum_{x } \int_{\Omega_x}  \big[\sqrt{f \left( \eta^{x,x+r} \right) } - \sqrt{f \left( \eta \right) } \big]^2 d \nu_{b} \lesssim    m^{\gamma} 3 \mcb   D (\sqrt{f}, \nu_{b} ) \lesssim   |r|^{\gamma} \mcb D (\sqrt{f}, \nu_{b} ),
\end{align*}
since $m \leq |r|$. 
\end{proof}
Now we state and prove the second step described above. The strategy is similar to the one used to prove Lemma 5.7 in \cite{MR4099999}.
\begin{lem} \label{lemrep2}
Assume $(\Phi_{n})_{n \geq 1}: [0,T] \times \mathbb{R} \rightarrow \mathbb{R}$ satisfies \eqref{boundrep} and denote $\ell(\varepsilon,n):=\varepsilon n^{\frac{\gamma}{2}}$. Then for every $t \in [0,T]$, it holds  
	\begin{equation} \label{eqlemrep2}
		\limsup_{\varepsilon \rightarrow 0^{+}}\limsup_{n \rightarrow \infty}\mathbb{E}_{\mu_n} \Big[ \Big| \int_{0}^{t}\frac{1}{n}\sum_{x } \Phi_n(s, \tfrac{x}{n} ) \overleftarrow{\eta}_{s}^{\ell}(x) [\eta_{s}(x+1)-\overrightarrow{\eta}_{s}^{\varepsilon n}(x+1)]ds \Big| \Big] =0.
	\end{equation}
\end{lem}
\begin{proof}
By entropy's inequality, \eqref{defcb}  and Jensen's inequality, we can bound the previous  expectation from above by
	\begin{equation} \label{rl3}
		\frac{C_b}{B} + \frac{1}{nB}\log \mathbb{E}_{\nu_{b}} \Big[ \exp \Big( Bn  \Big| \int_{0}^{t}\frac{1}{n}\sum_{x } \Phi_n(s, \tfrac{x}{n} ) \overleftarrow{\eta}_{s}^{\ell}(x) [\eta_{s}(x+1)-\overrightarrow{\eta}_{s}^{\varepsilon n}(x+1)]ds\Big)\Big] ,
	\end{equation}
for every $B>0$.  Since $e^{|u|} \leq e^{u}+e^{-u}$ and \eqref{ineqlimsup}, by Feynman-Kac's formula, we can bound last expression by 
\begin{equation}\label{expect1b}
\frac{C_b}{B} + T \sup_{f} \Big \{ \Big|  \int_{\Omega}\frac{1}{n}\sum_{x } \Phi_n(s, \tfrac{x}{n} ) \overleftarrow{\eta}^{\ell}(x) [\eta(x+1)-\overrightarrow{\eta}^{\varepsilon n}(x+1) ]f(\eta)d\nu_{b}  \Big|  + \frac{n^{\gamma-1}}{B}\langle \mcb L\sqrt{f}, \sqrt{f} \rangle_{\nu_{b}} \Big\} ds,
\end{equation}
where the supremum above is carried over all densities $f$ with respect to $\nu_{b}$. Since
\begin{align*}
\eta(x+1)-\overrightarrow{\eta}^{\varepsilon n}(x+1) = \displaystyle \frac{1}{\varepsilon n}\sum_{r=1}^{\varepsilon n} [\eta(x+1)-\eta(x+1+r)],
\end{align*}
we can rewrite the leftmost term inside the supremum in \eqref{expect1b} as
\begin{equation}
\Big| \frac{1}{\varepsilon n^2}\sum_{x } \Phi_n(s, \tfrac{x}{n} ) \sum_{r=1}^{\varepsilon n} \int_{\Omega}\overleftarrow{\eta}^{\ell}(x)[\eta(x+1)-\eta(x+1+r) ]f(\eta) d\nu_{b} \Big|.
\end{equation}
Writing $f(\eta) = \frac{1}{2}f(\eta)+\frac{1}{2}f(\eta)$, making the change of variables $\eta \mapsto \tilde{\eta}:= \eta^{x+1,x+1+r}$ in one of the integrals and using Remark \ref{revmeas}, we get that  last expression equals to 
\begin{equation}\label{expr3b}
\Big|\frac{1}{2 \varepsilon n^2}\sum_{x} \Phi_n(s, \tfrac{x}{n} ) \sum_{r=1}^{\varepsilon n} \int_{\Omega}\overleftarrow{\eta}^{\ell}(x) [\eta(x+1)-\eta(x+1+r) ] [f(\eta)-f(\eta^{x+1,x+1+r})] d\nu_{b} \Big|.
\end{equation}
We observe that $\overleftarrow{\eta}^{\ell}(x)$ is invariant under the change of variables $\eta \mapsto \tilde{\eta}$.  To treat the last display, we note that we are in a situation similar to the one used in Lemma 5.7 of \cite{MR4099999}, in which we need to exchange the particles in the bond $\{x+1,x+1+r\}$. 
First, for every $x \in \mathbb{Z}$, we denote the set of configurations that have at least two particles in $\{ x - \ell, \ldots, x-1 \}$ by $\Omega_1(x) := \left\{\eta \in \Omega: \overleftarrow{\eta}^{\ell}(x)\geq \frac{2}{\ell}\right\}$. Thus, we can bound \eqref{expr3b} by
\begin{equation}\label{expr3b1}
\Big|\frac{1}{2 \varepsilon n^2}\sum_{x} \Phi_n(s, \tfrac{x}{n} ) \sum_{r=1}^{\varepsilon n} \int_{ \Omega - \Omega_1(x)  } \overleftarrow{\eta}^{\ell}(x) [\eta(x+1)-\eta(x+1+r) ] [f(\eta)-f(\eta^{x+1,x+1+r})] d\nu_{b} \Big|
\end{equation}
\begin{equation}\label{expr3b2}
+\Big|\frac{1}{2 \varepsilon n^2}\sum_{x} \Phi_n(s, \tfrac{x}{n} ) \sum_{r=1}^{\varepsilon n} \int_{\Omega_1(x)}\overleftarrow{\eta}^{\ell}(x) [\eta(x+1)-\eta(x+1+r) ] [f(\eta)-f(\eta^{x+1,x+1+r})] d\nu_{b} \Big|.
\end{equation}
 Since $\frac{1}{n}\sum_{x} | \Phi_n(s, \tfrac{x}{n} ) |$ is bounded, $|\eta( \cdot )|\leq 1$ and $f$ is a density with respect to $\nu_{b}$,\eqref{expr3b1}  is bounded from above by a constant times $\frac{1}{\ell}$. Due to our choice of $\ell$ and since $\gamma>0$, \eqref{expr3b1} vanishes as $n$ goes to infinity. It remains to examine \eqref{expr3b2}, where we want to go from $\eta_{0,x,r}:=\eta$ to $\eta^{x+1,x+1+r}$.  The strategy is the following: for any configuration $ \eta \in \Omega_1(x)$, denote by $x_1$ and $x_2$ the position of the particles inside the box $\{x-\ell, \ldots, x-1 \}$ closest to the site $x+1$. With at most $2 \ell$ nearest-neighbor jumps, we can move the particles at $x_1$ and $x_2$ to $x$ and $x-1$, creating a group of at least $2$ particles in consecutive sites. We denote this configuration by $\eta_{1,x,r}$. Then, we exchange the particles in the bond $\{x+1,x+1+r\}$, following the procedure described in the proof of the Moving Particle Lemma. At this point, our configuration is  $\eta_{2,x,r}:=(\eta_{1,x,r})^{x+1,x+1+r}$. Finally, we use nearest-neighbor jumps in order to bring the particles at $x$ and $x-1$ back to their initial positions $ x_1$ and $x_2$, respectively. We observe that our configuration now is exactly $\eta_{3,x,r}:=\eta^{x+1,x+1+r}$. 
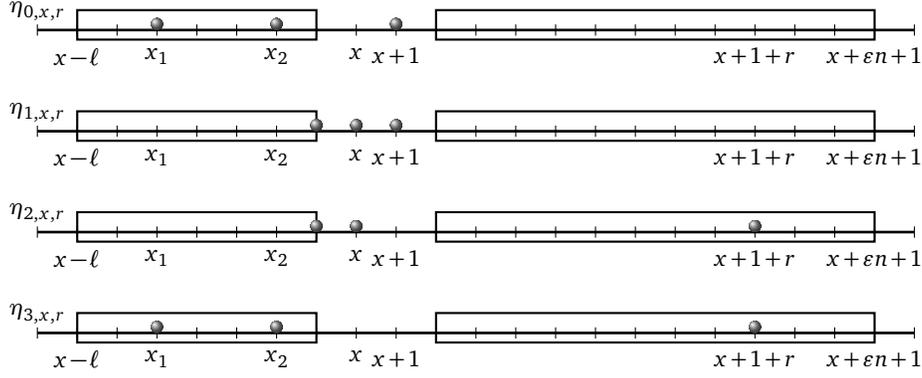
\begin{figure}[H]  %%%%%%%%%%%%%%%%%%%% Substituido figure* por figure  e h! por H
	\begin{center}
		\begin{tikzpicture}[scale=0.53]
			%------------------------------------------------------
			%configuração inicial
			\draw [line width=1] (-9,10.5) -- (13,10.5) ; %semi-reta com seta <-
			\foreach \x in  {-9,-8,-7,-6,-5,-4,-3,-2,-1,0,1,2,3,4,5,6,7,8,9,10,11,12,13} %marca os traços 
			\draw[shift={(\x,10.5)},color=black, opacity=1] (0pt,4pt) -- (0pt,-4pt) node[below] {};
			\draw[] (-2.8,10.5) node[] {};
			
			%colocando os indices
					\draw[] (-9,10.5) node[above] {\footnotesize{$\eta_{0,x,r}$}};
			\draw[] (-8,10.3) node[below] {\footnotesize{$x-\ell$}};
			\draw[] (-6,10.3) node[below] {\footnotesize{$x_1$}};
			\draw[] (-3,10.3) node[below] {\footnotesize{$x_2$}};
			\draw[] (-1,10.3) node[below] {\footnotesize{$x$}};
			\draw[] (0,10.3) node[below] {\footnotesize{$x+1$}};
		%	\draw[] (4,10.3)  node[below] {\footnotesize{$z_{1j}$}};
			\draw[] (9,10.3)  node[below] {\footnotesize{$x+1+r$}};
			\draw[] (12,10.3) node[below] {\footnotesize{$x+ \varepsilon n+1$}};
			
			%desenhando caixa
			\draw[thick] (-8, 10.26) rectangle (-2, 11);
			\draw[thick] (1, 10.26) rectangle (12, 11);
			
			%partículas da conf inicial
			\shade[shading=ball, ball color=black!50!] (-6,10.65) circle (.15);
			\shade[shading=ball, ball color=black!50!] (-3,10.65) circle (.15);
			\shade[shading=ball, ball color=black!50!] (0,10.65) circle (.15);
		\end{tikzpicture} 
	\end{center}
	\begin{center}
		\begin{tikzpicture}[scale=0.53]
			%------------------------------------------------------
			%configuração inicial
			\draw [line width=1] (-9,10.5) -- (13,10.5) ; %semi-reta com seta <-
			\foreach \x in  {-9,-8,-7,-6,-5,-4,-3,-2,-1,0,1,2,3,4,5,6,7,8,9,10,11,12,13}
			\draw[shift={(\x,10.5)},color=black, opacity=1] (0pt,4pt) -- (0pt,-4pt) node[below] {};
			\draw[] (-2.8,10.5) node[] {};
			
			%colocando os indices
					\draw[] (-9,10.5) node[above] {\footnotesize{$\eta_{1,x,r}$}};
			\draw[] (-8,10.3) node[below] {\footnotesize{$x-\ell $}};
			\draw[] (-6,10.3) node[below] {\footnotesize{$x_1$}};
			\draw[] (-3,10.3) node[below] {\footnotesize{$x_2$}};
			\draw[] (-1,10.3) node[below] {\footnotesize{$x$}};
			\draw[] (0,10.3) node[below] {\footnotesize{$x+1$}};
	%		\draw[] (4,10.3)  node[below] {\footnotesize{$z_{1j}$}};
			\draw[] (9,10.3)  node[below] {\footnotesize{$x+1+r$}};
			\draw[] (12,10.3) node[below] {\footnotesize{$x+ \varepsilon n+1$}};
			
			%desenhando caixa
			\draw[thick] (-8, 10.26) rectangle (-2, 11);
			\draw[thick] (1, 10.26) rectangle (12, 11);
			
			%partículas da conf inicial
			\shade[shading=ball, ball color=black!50!] (-2,10.65) circle (.15);
			\shade[shading=ball, ball color=black!50!] (-1,10.65) circle (.15);
			\shade[shading=ball, ball color=black!50!] (0,10.65) circle (.15);
		\end{tikzpicture} 
	\end{center}
	\begin{center}
		\begin{tikzpicture}[scale=0.53]
			%------------------------------------------------------
			%configuração inicial
			\draw [line width=1] (-9,10.5) -- (13,10.5) ; %semi-reta com seta <-
			\foreach \x in  {-9,-8,-7,-6,-5,-4,-3,-2,-1,0,1,2,3,4,5,6,7,8,9,10,11,12,13}
			\draw[shift={(\x,10.5)},color=black, opacity=1] (0pt,4pt) -- (0pt,-4pt) node[below] {};
			\draw[] (-2.8,10.5) node[] {};
		
					\draw[] (-9,10.5) node[above] {\footnotesize{$\eta_{2,x,r}$}};
			\draw[] (-8,10.3) node[below] {\footnotesize{$x-\ell $}};
			\draw[] (-6,10.3) node[below] {\footnotesize{$x_1$}};
			\draw[] (-3,10.3) node[below] {\footnotesize{$x_2$}};
			\draw[] (-1,10.3) node[below] {\footnotesize{$x$}};
			\draw[] (0,10.3) node[below] {\footnotesize{$x+1$}};
	%		\draw[] (4,10.3)  node[below] {\footnotesize{$z_{1j}$}};
			\draw[] (9,10.3)  node[below] {\footnotesize{$x+1+r$}};
			\draw[] (12,10.3) node[below] {\footnotesize{$x+ \varepsilon n+1$}};
			\draw[thick] (-8, 10.26) rectangle (-2, 11);
			\draw[thick] (1, 10.26) rectangle (12, 11);
			\shade[shading=ball, ball color=black!50!] (-2,10.65) circle (.15);
			\shade[shading=ball, ball color=black!50!] (-1,10.65) circle (.15);
			\shade[shading=ball, ball color=black!50!] (9,10.65) circle (.15);
		\end{tikzpicture} 
	\end{center}
	\begin{center}
		\begin{tikzpicture}[scale=0.53]
			%------------------------------------------------------
			%configuração inicial
			\draw [line width=1] (-9,10.5) -- (13,10.5) ; %semi-reta com seta <-
			\foreach \x in  {-9,-8,-7,-6,-5,-4,-3,-2,-1,0,1,2,3,4,5,6,7,8,9,10,11,12,13}
			\draw[shift={(\x,10.5)},color=black, opacity=1] (0pt,4pt) -- (0pt,-4pt) node[below] {};
			\draw[] (-2.8,10.5) node[] {};
					\draw[] (-9,10.5) node[above] {\footnotesize{$\eta_{3,x,r}$}};
			\draw[] (-8,10.3) node[below] {\footnotesize{$x-\ell $}};
			\draw[] (-6,10.3) node[below] {\footnotesize{$x_1$}};
			\draw[] (-3,10.3) node[below] {\footnotesize{$x_2$}};
			\draw[] (-1,10.3) node[below] {\footnotesize{$x$}};
			\draw[] (0,10.3) node[below] {\footnotesize{$x+1$}};
		%	\draw[] (4,10.3)  node[below] {\footnotesize{$z_{1j}$}};
			\draw[] (9,10.3)  node[below] {\footnotesize{$x+1+r$}};
			\draw[] (12,10.3) node[below] {\footnotesize{$x+ \varepsilon n+1$}};
			\draw[thick] (-8, 10.26) rectangle (-2, 11);
			\draw[thick] (1, 10.26) rectangle (12, 11);
			\shade[shading=ball, ball color=black!50!] (-6,10.65) circle (.15);
			\shade[shading=ball, ball color=black!50!] (-3,10.65) circle (.15);
			\shade[shading=ball, ball color=black!50!] (9,10.65) circle (.15);
		\end{tikzpicture} 
	\end{center}
	\caption{Strategy used to exchange particles at the bond $\{x,x+1+r\}$}
	\label{figure-path}
\end{figure}
 Hence, we can write
\begin{align*}
f(\eta) - f(\eta^{x+1, x+1+r}) = \sum_{k=0}^2 [f(\eta_{k,x,r}) - f(\eta_{k+1,x,r}) ].
\end{align*} 
Then we can bound \eqref{expr3b2} from above by the sum of
\begin{equation} \label{expr4b}
\Big| \frac{1}{2 \varepsilon n^2}\sum_{x } \Phi_n(s, \tfrac{x}{n} ) \sum_{r=1}^{\varepsilon n} \int_{\Omega_1(x)}\overleftarrow{\eta}^{\ell}(x) [\eta(x+1)-\eta(x+1+r) ] [f(\eta) - f( \eta_{1,x,r}) ] d\nu_{b} \Big|,
\end{equation}
\begin{equation} \label{expr5b}
\Big| \frac{1}{2 \varepsilon n^2}\sum_{x } \Phi_n(s, \tfrac{x}{n} ) \sum_{r=1}^{\varepsilon n} \int_{\Omega_1(x)}\overleftarrow{\eta}^{\ell}(x) [\eta(x+1)-\eta(x+1+r) ] [  f ( \eta_{2,x,r}) - f( \eta_{3,x,r} ) ] d\nu_{b} \Big|,
\end{equation} 
\begin{equation} \label{expr7b}
\Big| \frac{1}{2 \varepsilon n^2}\sum_{x} \Phi_n(s, \tfrac{x}{n} ) \sum_{r=1}^{\varepsilon n} \int_{\Omega_1(x)}\overleftarrow{\eta}^{\ell}(x) [\eta(x+1)-\eta(x+1+r) ] [ f  ( \eta_{1,x,r} ) - f ( \eta_{2,x,r} )  ]  d\nu_{b} \Big|.
\end{equation}
We observe that both \eqref{expr4b} and \eqref{expr5b} deal with nearest-neighbor jumps and we can estimate both expressions in the same way. Let us examine \eqref{expr4b}. We can write 
\begin{align*}
f(\eta) - f( \eta_{1,x,r}) = \sum_{i\in I_{1,x}^{NN}} [f(\eta^{(i-1)})-f(\eta^{(i)}) ],
\end{align*}
where $I_{1,x}^{NN}$ are the set of bonds in which we use nearest-neighbor jumps. Note that for two nonnegative numbers $x$ and $y$, it holds  $x-y = \left(\sqrt{x}-\sqrt{y}\right)\left(\sqrt{x}+\sqrt{y}\right)$. Thus, combining this identity with Young's inequality, we can bound \eqref{expr4b} from above by
\begin{align*} 
&\Big| \frac{1}{4 \varepsilon n^2}\sum_{x } \Phi_n(s, \tfrac{x}{n} ) \sum_{r=1}^{\varepsilon n} \sum_{i\in I_{1,x}^{NN}}\frac{1}{A_{NN}} \int_{\Omega_1(x)} \overleftarrow{\eta}^{\ell}(x)^2 [\eta(x+1)-\eta(x+1+r)]^2 \big[ \sqrt{f(\eta^{(i-1)})}+\sqrt{f(\eta^{(i)})}\big]^2  d\nu_{b} \Big|\\
+& \Big| \frac{1}{4 \varepsilon n^2}\sum_{x } \Phi_n(s, \tfrac{x}{n} )  \sum_{i \in I_{1,x}^{NN}} A_{NN} \int_{\Omega_1(x)} \sum_{r=1}^{\varepsilon n} \big[\sqrt{f(\eta^{(i-1)})}-\sqrt{f(\eta^{(i)})}\big]^2 d\nu_{b}  \Big| \\
\leq & \frac{1}{4 \varepsilon n} \frac{1}{n} \sum_{x } | \Phi_n(s, \tfrac{x}{n} )| \sum_{r=1}^{\varepsilon n} \frac{8 \ell}{A_{NN}} + \frac{ A_{NN} \| \Phi_n \|_{\infty} }{4n}  \int_{\Omega}\sum_{x } \sum_{i \in I_{1,x}^{exc}} \big[\sqrt{f(\eta^{(i-1)})}-\sqrt{f(\eta^{(i)})}\big]^2 d\nu_{b},
\end{align*}
for any $A_{NN}>0$. Above we used the fact  that $I_{1,x}^{NN}$ has at most $2 \ell$ bonds, $\eta(x) \leq 1$ for $x\in\mathbb{Z}$, and $f$ is a density with respect to $\nu_{b}$. We observe that in the double summation inside the integral over $\Omega$ above, every bond appears at most $2 \ell$ times.  From \eqref{D_S} and \eqref{boundrep} we can bound \eqref{expr4b} from above by a constant times
	\begin{align} \label{expr8b}
		\frac{ \ell M_1}{A_{NN}} + \frac{A_{NN} M_2}{n }\ell \mcb D_{NN} (\sqrt{f},\nu_{b}),
	\end{align}
for every $A_{NN} > 0$. With the same reasoning, we can bound \eqref{expr5b} from above by a constant times
\begin{align} \label{expr9b}
\frac{ \ell M_1}{A_{NN}} + \frac{A_{NN}M_2}{n }\ell \mcb D_{NN} (\sqrt{f},\nu_{b}).
\end{align} 
We observe that \eqref{expr7b} deals mostly with long jumps. With a similar reasoning as we did with \eqref{expr4b}, we can  bound \eqref{expr7b} from above  by
\begin{align*}
& \Big| \frac{1}{4 \varepsilon n^2} \sum_{x } \Phi_n(s, \tfrac{x}{n} ) \sum_{r=1}^{\varepsilon n} \frac{1}{A} \int_{\Omega_1(x)}[\overleftarrow{\eta}^{\ell}(x)]^2 [\eta(x+1)-\eta(x+1+r) ]^2 \big[ \sqrt{f( \eta_{1,x,r} )}  + \sqrt{ f ( \eta_{2,x,r} )  }  \big]^2  d\nu_{b} \Big| \\
+ & \Big| \frac{A}{4 \varepsilon n^2}\sum_{x } \Phi_n(s, \tfrac{x}{n} ) \sum_{r=1}^{\varepsilon n}  \int_{\Omega_1(x)}  \big[ \sqrt{f( \eta_{1,x,r} )}  - \sqrt{ f \big( ( \eta_{1,x,r} )^{x+1,x+1+r}  \big) }  \big]^2  d\nu_{b} \Big| \\
 \leq & \frac{M_1}{A} + \frac{A  M_2 }{4 \varepsilon n^2} \sum_{r=1}^{\varepsilon n} \sum_{x } \int_{\Omega_1(x)}  \big[ \sqrt{f( \eta_{1,x,r})}  - \sqrt{ f \big( ( \eta_{1,x,r})^{x+1,x+1+r}  \big) }  \big]^2  d \nu_{b} ,
\end{align*} 
 for every $A>0$. Above we made use of \eqref{boundrep} and Remark \ref{revmeas}. From the Moving Particle Lemma, we can bound \eqref{expr7b} from above  by a constant times
 	\begin{align} \label{expr10b}
 		\frac{M_1}{A} + \frac{A M_2 }{4 \varepsilon n^2} \sum_{r=1}^{\varepsilon n} r^{\gamma} \mcb D(\sqrt{f},\nu_{b}) \leq \frac{M_1}{A} + AM_2 \varepsilon^{\gamma} n^{\gamma-1} \mcb  D (\sqrt{f},\nu_{b}),
 	\end{align}
 for every $A>0$. Therefore, taking  $A_{NN}= \frac{ n^{\gamma} }{4 B M_2 \ell}$  in \eqref{expr8b} and  \eqref{expr9b}, $A= \frac{1}{2 B M_2 \varepsilon^{\gamma} }$ in \eqref{expr10b}, from Lemma \ref{bound-Dirichlet form} we conclude that \eqref{expect1b} is bounded from above by a constant, times
$$\frac{1}{B} + T B M_1 M_2 \Big(\frac{ \ell^2}{n^{\gamma}} + \varepsilon^{\gamma} \Big).$$
Choosing $B=\varepsilon^{-\frac{\gamma}{2}}$, due to our choice for $\ell$, the previous expression vanishes for any $\gamma<2$ when we take first $n \rightarrow \infty$ and then $\varepsilon \rightarrow 0^{+}$.
\end{proof}
Finally we state and prove the third step. The strategy is similar to the one used to prove Lemma 5.8 in \cite{MR4099999}.
\begin{lem} \label{lemrep3}
Assume $(\Phi_{n})_{n \geq 1}: [0,T] \times \mathbb{R} \rightarrow \mathbb{R}$ satisfies \eqref{boundrep} and denote $\ell(\varepsilon,n):=\varepsilon n^{\frac{\gamma}{2}}$. Then for every $t \in [0,T]$, it holds
	\begin{equation} \label{eqlemrep3}
		\limsup_{\varepsilon \rightarrow 0^{+}}\limsup_{n \rightarrow \infty}\mathbb{E}_{\mu_n} \Big[ \Big| \int_{0}^{t}\frac{1}{n}\sum_{x  } \Phi_{n}(s,\tfrac{x}{n} ) [ \overleftarrow{\eta}_{s}^{\ell}(x) - \overleftarrow{\eta}_{s}^{\varepsilon n}(x)] \overrightarrow{\eta}_{s}^{\varepsilon n}(x+1) ds \Big| \Big] =0.
	\end{equation}
\end{lem}
\begin{proof}
After similar steps to the one ones performed in the beginning of the proof of Lemma \ref{lemrep2} we can bound the previous expectation by
\begin{equation}\label{expect1c}
\frac{C_b}{B} + T\sup_{f} \Big\{ \Big|  \int_{\Omega}\frac{1}{n}\sum_{x } \Phi_{n}(s,\tfrac{x}{n} ) [ \overleftarrow{\eta}^{\ell}(x) - \overleftarrow{\eta}^{\varepsilon n}(x) ] \overrightarrow{\eta}^{\varepsilon n}(x+1) f(\eta) d\nu_{b} \Big| + \frac{n^{\gamma-1}}{B}\langle \mcb L\sqrt{f}, \sqrt{f} \rangle_{\nu_{b}} \Big\} ds,
\end{equation}
where the supremum above is carried over all densities $f$ with respect to $\nu_{b}$. 
In the same way it was done in the proof of Lemma 5.8 in \cite{MR4099999}, we write $\varepsilon n = m \ell$ and observe that
	\begin{align*}
	\forall x \in \mathbb{Z}, \quad	\overleftarrow{\eta}^{\ell}(x) - \overleftarrow{\eta}^{\varepsilon n}(x)  =  \frac{1}{m \ell} \sum_{j=1}^{m-1} \sum_{z = x-\ell}^{x-1} [ \eta(z) - \eta(z - j \ell) ].
	\end{align*}
Then we can rewrite the leftmost term in the supremum in \eqref{expect1c} as
 \begin{equation*} 
\Big| \frac{1}{m \ell n}\sum_{x} \Phi_{n}(s,\tfrac{x}{n} ) \sum_{j=1}^{m-1}   \frac{1}{\ell} \sum_{z = x-\ell}^{x-1} \int_{\Omega} [ \eta(z) - \eta(z - j \ell) ] \overrightarrow{\eta}^{\varepsilon n}(x+1) f(\eta) d\nu_{b} \Big|.
 \end{equation*}
 Writing $f(\eta) = \frac{1}{2}f(\eta)+\frac{1}{2}f(\eta)$, making the change of variables $\eta \mapsto \tilde{\eta}:=\eta^{z,z-j \ell}$ in one of the integrals and using Remark \ref{revmeas}, last expression is equal to
\begin{equation}\label{expr3c}
\Big| \frac{1}{2m \ell n}\sum_{x} \Phi_{n}(s,\tfrac{x}{n} ) \sum_{j=1}^{m-1}    \sum_{z = x-\ell}^{x-1}  \int_{\Omega} [ \eta(z) - \eta(z - j \ell) ] \overrightarrow{\eta}^{\varepsilon n}(x+1) [f(\eta) - f(\eta^{z,z- j \ell}) ]  d\nu_{b} \Big|.
\end{equation}
Above we observe that $\overrightarrow{\eta}^{\varepsilon n}(x+1)$ is invariant under the change of variables $\eta \mapsto \tilde{\eta}$, since $z \leq x-1$. Now we need to examine \eqref{expr3c}. Our goal is to exchange the particles in the bond $\{z,z - j \ell\}$. For every $x \in \mathbb{Z}$ and every $j \in \{1, \ldots, m-1\}$, we denote
\begin{align*}
\Omega_2^j(x) := \Big\{\eta \in \Omega: \overleftarrow{\eta}^{\ell}(x)\geq \frac{2}{\ell}\Big\} \cup \Big\{\eta \in \Omega: \overleftarrow{\eta}^{\ell}(x- j \ell)\geq \frac{2}{\ell}\Big\}.
\end{align*}
Thus, we can bound \eqref{expr3c} by
\begin{equation}\label{expr3c1}
\Big| \frac{1}{2m \ell n}\sum_{x} \Phi_{n}(s,\tfrac{x}{n} ) \sum_{j=1}^{m-1}    \sum_{z = x-\ell}^{x-1}  \int_{\Omega - \Omega_2^j(x)} [ \eta(z) - \eta(z - j \ell) ] \overrightarrow{\eta}^{\varepsilon n}(x+1) [f(\eta) - f(\eta^{z,z- j \ell}) ]  d\nu_{b} \Big|
\end{equation}
\begin{equation}\label{expr3c2}
+\Big| \frac{1}{2m \ell n}\sum_{x} \Phi_{n}(s,\tfrac{x}{n} ) \sum_{j=1}^{m-1}    \sum_{z = x-\ell}^{x-1}  \int_{\Omega_2^j(x)} [ \eta(z) - \eta(z - j \ell) ] \overrightarrow{\eta}^{\varepsilon n}(x+1) [f(\eta) - f(\eta^{z,z- j \ell}) ]  d\nu_{b} \Big|.
\end{equation}
Since  $\frac{1}{n}\sum_{x\in\mathbb{Z}} | \Phi_{n}(s,\tfrac{x}{n} )|$ is bounded, $|\eta( \cdot )|\leq 1$ and $f$ is a density with respect to $\nu_{b}$, \eqref{expr3c1} is bounded from above by a constant times $\frac{1}{\ell}$. Due to our choice of $\ell$ and since  $\gamma>0$ , \eqref{expr3c1} vanishes as $n$ goes to infinity.
It remains to deal with \eqref{expr3c2}, where we want to go from $\eta_{0,x,j,z}:=\eta$ to $\eta^{z,z-j \ell}$. If $\overleftarrow{\eta}^{\ell}(x)\geq \frac{2}{\ell}$,  the strategy is the following: for any configuration $ \eta \in \Omega_1(x)$, denote by $x_1$ and $x_2$ the position of the particles inside the box $\{x-\ell, \ldots, x-1 \}$ closest to $z$. With at most $2 \ell$ nearest-neighbor jumps, we can move the particles at $x_1$ and $x_2$ to $z-1$ and $z-2$. On the other hand, if $\eta\in\Omega_2^j$ and $\overleftarrow{\eta}^{\ell}(x)\leq \frac{1}{\ell}$ then necessarily we have $\overleftarrow{\eta}^{\ell}(x-j \ell)\geq \frac{2}{\ell}$; in this case, denote by $x_1$ and $x_2$ the position of the particles inside the box $\{x- j \ell - \ell, \ldots, x- j \ell -1 \}$ closest to $z-j \ell$. With at most $2 \ell$ nearest-neighbor jumps, we can move the particles at $x_1$ and $x_2$ to $z-j \ell -1$ and $z-j \ell -2$. In both cases, we denote the configuration with the group of at least two particles in consecutive sites next to $z$ (resp. $z-j \ell$) by $\eta_{1,x,j,z}$. Then, we exchange the particles in the bond $\{z, z - j \ell\}$, following the procedure described in the proof of the Moving Particle Lemma. At this point, our configuration is  $\eta_{2,x,j,z}:=(\eta_{1,x,j,z})^{z,z-j \ell}$. Finally, we use nearest-neighbor jumps in order to bring the two auxiliary particles back to their initial positions $x_1$ and $x_2$. We observe that our configuration now is exactly $\eta_{3,x,j,z}:=\eta^{z,z-j \ell}$.
 Hence, we can write
\begin{align*}
f(\eta) - f(\eta^{z, z-j \ell}) = \sum_{k=0}^2 [f(\eta_{k,x,j,z}) - f(\eta_{k+1,x,j,z}) ].
\end{align*} 
In this way, we can bound \eqref{expr3c2} from above by the sum of
\begin{align} \label{expr4c}
\Big| \frac{1}{2m \ell n}\sum_{x} \Phi_{n}(s,\tfrac{x}{n} ) \sum_{j=1}^{m-1}  \int_{\Omega_2^j(x)}  \sum_{z = x-\ell}^{x-1} [ \eta(z) - \eta(z - j \ell) ] \overrightarrow{\eta}^{\varepsilon n}(x+1) [f(\eta) - f( \eta_{1,x,j,z} ) ]  d\nu_{b} \Big|,
\end{align}
\begin{align} \label{expr5c}
\Big| \frac{1}{2m \ell  n}\sum_{x } \Phi_{n}(s,\tfrac{x}{n} ) \sum_{j=1}^{m-1}  \int_{\Omega_2^j(x)}  \sum_{z = x-\ell}^{x-1} [ \eta(z) - \eta(z - j \ell) ] \overrightarrow{\eta}^{\varepsilon n}(x+1) [f ( \eta_{2,x,j,z} )  - f( \eta_{3,x,j,z} )  ]  d\nu_{b} \Big|,
\end{align}
\begin{align} \label{expr6c}
\Big| \frac{1}{2 m \ell n}\sum_{x } \Phi_{n}(s,\tfrac{x}{n} ) \sum_{j=1}^{m-1}  \int_{\Omega_2^j(x)}  \sum_{z = x-\ell}^{x-1} [ \eta(z) - \eta(z - j \ell) ] \overrightarrow{\eta}^{\varepsilon n}(x+1) [f( \eta_{1,x,j,z} ) - f \big( \eta_{2,x,j,z} ) ]  d\nu_{b} \Big|.
\end{align}
We observe that both \eqref{expr4c} and \eqref{expr5c} deal with nearest-neighbor jumps and we can estimate both expressions in the same way. With an analogous procedure used to estimate \eqref{expr4b}, from \eqref{D_S} and \eqref{boundrep} we can bound both \eqref{expr4c} and \eqref{expr5c} from above by a constant times
	\begin{align} \label{expr10c}
		\frac{ \ell M_1}{A_{NN}} + \frac{A_{NN}M_2}{n }\ell \mcb D_{NN} (\sqrt{f},\nu_{b}),
	\end{align}
for every $A_{NN}>0$. We observe that \eqref{expr6c} deals with mostly with long jumps. With a similar reasoning as we did with \eqref{expr7b}, we can bound \eqref{expr6c} from above by
\begin{align*}
&\Big| \frac{1}{4 m \ell n A }\sum_{x } \Phi_{n}(s,\tfrac{x}{n} ) \sum_{j=1}^{m-1}    \sum_{z = x-\ell}^{x-1} \int_{\Omega_2^j(x)} [ \eta(z) - \eta(z - j \ell) ]^2  \big[\sqrt{f( \eta_{1,x,j,z})}  + \sqrt{ f ( \eta_{2,x,j,z} )  }  \big]^2 d\nu_{b} \Big| \\
+ &\Big| \frac{A}{4 m \ell n }\sum_{x } \Phi_{n}(s,\tfrac{x}{n} ) \sum_{j=1}^{m-1}  \int_{\Omega_2^j(x)}  \sum_{z = x-\ell}^{x-1}  [\overrightarrow{\eta}^{\varepsilon n}(x+1)]^2 \big[\sqrt{f( \eta_{1,x,j,z})}  - \sqrt{ f \big( ( \eta_{1,x,j,z})^{z,z - j \ell} \big)}  \big]^2  d\nu_{b} \Big| \\
 \leq & \frac{M_1}{A} + \frac{A M_2 }{4 \varepsilon n^2} \sum_{j=1}^{m-1} \sum_{x }  \sum_{z = x-\ell}^{x-1} \int_{\Omega_2^j(x)}  \big[ \sqrt{f(\eta_{1,x,j,z})}  - \sqrt{ f \big( (\eta_{1,x,j,z})^{z,z - j \ell} \big) }  \big]^2  d \nu_{b} ,
\end{align*}
 for every $A> 0$. Above we used \eqref{boundrep} and Remark \ref{revmeas}. From the Moving Particle Lemma, we can bound \eqref{expr6c} by a constant times
	\begin{align} \label{expr12c}
		\frac{M_1}{A} + \frac{AM_2 }{4 \varepsilon n^2} \ell \sum_{j=1}^{m-1} (j \ell )^{\gamma} \mcb D (\sqrt{f},\nu_{b}) \leq \frac{M_1}{A} + A M_2 \varepsilon^{\gamma} n^{\gamma-1} \mcb  D (\sqrt{f},\nu_{b}),
	\end{align}
for every $A>0$. Therefore, choosing $A_{NN}= \frac{ n^{\gamma} }{4 B M_2 \ell}$  in \eqref{expr10c} and $A= \frac{1}{4 B M_2 \varepsilon^{\gamma} }$ in \eqref{expr12c}, from \eqref{bound-Dirichlet form} we conclude that \eqref{expect1c} is bounded from above by a constant, times
$$\frac{1}{B} + T B M_1 M_2 \Big(\frac{ \ell^2}{n^{\gamma}} + \varepsilon^{\gamma} \Big).$$
Choosing $B=\varepsilon^{-\frac{\gamma}{2}}$, due to our choice for $\ell$, the previous expression vanishes for any $\gamma<2$ when we take first $n \rightarrow \infty$ and then $\varepsilon \rightarrow 0^{+}$.
\end{proof}

\appendix
\section{Discrete convergences} \label{secdiscconv}
\label{secuseres}

In this section we present the proof of Proposition \ref{convext}.
\begin{proof}[Proof of Proposition \ref{convext}]
Performing two Taylor expansions of second order on $G$ around $(s, \tfrac{x}{n})$ and $(s, \tfrac{y}{n})$, we can bound the double summation above by
\begin{align}
&n^{\gamma-3} \sum_{x }  \sup_{s \in [0,T]} | \Delta G (s, \chi_x) |   \sum_{y} p(y-x) + n^{\gamma-3} \sum_{y }  \sup_{s \in [0,T]} | \Delta G (s, \chi_y) |   \sum_{x} p(y-x) \label{ext1}  \\
+&n^{\gamma-2}  \sum_{x,y }  \sup_{s \in [0,T]} | \partial_u G(s, \tfrac{x}{n}) - \partial_u G(s, \tfrac{y}{n})  | p(y-x),  \label{ext2}
\end{align}
where $\chi_x \in ( \frac{x}{n}, \frac{x+1}{n} )$ for every $x \in \mathbb{Z}$ and $\chi_y \in ( \frac{y}{n}, \frac{y+1}{n} )$ for every $y \in \mathbb{Z}$. Since $\sup_{s \in [0,T]} |G(s,u)|=0$ if $|u| \geq b_G$,  \eqref{ext1} can be bounded by
$2 (2b_G + 3) \|\Delta G \|_{\infty} n^{\gamma-2},
$ 
which vanishes as $n$ goes to infinity, since $\gamma < 2$. Now we rewrite \eqref{ext2} as
\begin{align}
& n^{\gamma-2} \sum_{|x| > 2 b_G n} \sum_{|y| \leq  b_G n} \sup_{s \in [0,T]} | \partial_u G(s, \tfrac{x}{n}) - \partial_u G(s, \tfrac{y}{n})  | p(y-x)  \label{ext3} \\
+& n^{\gamma-2} \sum_{|x| \leq 2 b_G n} \sum_{y} \sup_{s \in [0,T]} | \partial_u G(s, \tfrac{x}{n}) - \partial_u G(s, \tfrac{y}{n})  | p(y-x)  \label{ext4}.
\end{align}
We can bound \eqref{ext3} by
\begin{align*}
n^{\gamma-2} 2 \|\partial_u G \|_{\infty} \sum_{|x| > 2 b_G n} \sum_{|y| \leq  b_G n} p(y-x) \lesssim n^{-1} \int_{|v| \leq b_G} \int_{|u| \geq 2 b_G} |u-v|^{-1-\gamma} du dv \lesssim n^{-1} (b_G)^{1-\gamma}, 
\end{align*}
which vanishes as $n$ goes to infinity. It remains to treat \eqref{ext4}. We bound it from above by
\begin{align}
& n^{\gamma-2} \sum_{|x| \leq 2 b_G n} \sum_{|y| > 3 b_G n} \sup_{s \in [0,T]} | \partial_u G(s, \tfrac{x}{n}) - \partial_u G(s, \tfrac{y}{n})  | p(y-x)  \label{ext5} \\
+& n^{\gamma-2} \sum_{|x| \leq 2 b_G n} \sum_{|y| \leq 3 b_G n} \sup_{s \in [0,T]} | \partial_u G(s, \tfrac{x}{n}) - \partial_u G(s, \tfrac{y}{n})  | p(y-x)  \label{ext6}.
\end{align}
We bound \eqref{ext5} from above by a constant times
\begin{align*}
& n^{\gamma-2} \sum_{|x| \leq 2 b_G n} \sum_{y > 3 b_G n} \| \partial_u G \|_{\infty} (y-x)^{-\gamma-1} \lesssim n^{-2} \sum_{|x| \leq 2 b_G n} \frac{1}{n} \sum_{y = 3 b_G n}^{\infty} \Big( \frac{y-x}{n} \Big)^{-\gamma-1} \\
\lesssim & \ n^{-2} \sum_{|x| \leq 2 b_G n} \Big( 3 b_G  - \frac{x}{n} \Big)^{-\gamma} \leq \ n^{-2} \sum_{|x| \leq 2 b_G n} (b_G)^{-\gamma} \lesssim n^{-1},
\end{align*}
which goes to zero as $n \rightarrow \infty$. Finally, observe that functions in $C_c^{\infty}(\mathbb{R})$ are globally Lipschitz and therefore $\delta$-H\"older on the compact space $[-3b_G, 3 b_G]$ for every $\delta \in [0,1]$. Since $G \in \mcb S$ and $\gamma \in (0,2)$ we can always choose $\delta \in [0,1] \cap (\gamma-1, \gamma)$ such that
	\begin{align*}
	\forall x,y \in [-3b_G n, 3 b_G n], \quad	\sup_{s \in [0,T]} | \partial_u G(s, \tfrac{x}{n}) - \partial_u G(s, \tfrac{y}{n})  | \leq C_{G,\delta} |x-y|^{\delta} n^{-\delta}.
	\end{align*} 
for some $C_{G,\delta} > 0$ (for example, we can choose $\delta=0$ for $\gamma \in (0,1)$, $\delta=1/2$ for $\gamma=1$ and $\delta=1$ for $\gamma \in (1,2)$). Then we can bound \eqref{ext6} by
\begin{align*}
& n^{\gamma-2} \sum_{|x| \leq 2 b_G n} \sum_{|y| \leq 2 b_G n} \sup_{s \in [0,T]} C_{G,\delta} |x-y|^{\delta} n^{-\delta} p(y-x) \\ % \lesssim  n^{\gamma-2-\delta} \sum_{|x| \leq 2 b_G n} \textcolor{red}{\sum_{y \neq z}}|x-y|^{\delta-\gamma-1} \\
\lesssim & n^{\gamma-2-\delta} \sum_{|x| \leq 2 b_G n} \sum_{z \neq 0} |z|^{\delta-\gamma-1} = 2 (2b_G +1) n^{\gamma-1-\delta} \sum_{z = 1}^{\infty} z^{\delta-\gamma-1}. 
\end{align*}
Since $\delta < \gamma$, the summation over $z$ in the last line is convergent and the expression in last display is of order $n^{\gamma-1-\delta}$, going to zero as $n$ goes to infinity (since $\delta > \gamma-1$). This ends the proof.
\end{proof}

\section{Properties of the fractional Laplacian}\label{ap:prop}

Here we state some results regarding the fractional Laplacian which we used earlier. Recall \eqref{space_functions}. Then for $G\in\mcb S$ we have 
\begin{align*} 
\forall (t,u) \in [0,T] \times \mathbb{R}, \quad	[- ( - \Delta)^{\gamma/2} G] (t,u) = \sum_{j=0}^{k} t^{j} [- ( - \Delta)^{\gamma/2} G_j](u).
\end{align*}
We begin with a classical result which is stated and proved in Section 2.2.1 of \cite{daoud}. In this Section $G^{(k)}$ denotes the $k-th$ derivative of $G \in C^{k}(\mathbb{R})$.
\begin{prop} \label{propdif}
We have $\big( [- (- \Delta)^{\gamma/2} ]G \big)^{(k)} \big)= [ - (- \Delta)^{\gamma/2}] G^{(k)}$, for every $G \in C_c^{\infty}(\mathbb{R})$ and every $k \geq 1$. In particular, for $G \in C_c^{\infty}(\mathbb{R})$, it holds $[ -(-\Delta)^{\gamma/2}G](s,  \cdot ) \in C^{\infty}(\mathbb{R})$, for every $s \in [0,T]$. 
\end{prop}
Next we present the proof of Proposition \ref{propL1}.
\begin{proof} [Proof of Proposition \ref{propL1}]
Fix $G \in C_c^{\infty}(\mathbb{R})$ and recall the definition of $b_G$ in \eqref{defbG}. Now define $H^G: \mathbb{R} \rightarrow \mathbb{R}$ by
\begin{equation*}
H^G(u):=
\begin{cases}
c_{\gamma} \int_{ -u-b_G}^{-u+b_G}  \frac{  \| G \|_{\infty} }{ w^{\gamma+1}} dw, \text{if} \; u < - 2b_G; \\
c_{\gamma}      \int_{ 0}^{3b_G}  \frac{  \|  \Delta G \|_{\infty}    }{ w^{\gamma-1}} dw, \text{if} \; - 2b_G \leq u \leq  2b_G; \\
c_{\gamma} \int_{ u-b_G}^{u+b_G}  \frac{ \| G \|_{\infty} }{ w^{\gamma+1}} dw, \text{if} \; u > 2b_G.
\end{cases}
\end{equation*}
From simple computations, we conclude that $H^G \in L^1(\mathbb{R}) \cap L^{\infty}(\mathbb{R})$. Then, it is enough  to prove that $|[-(- \Delta)^{\gamma/2} G ] (u )| \leq H^G(u)$, for every $u \in \mathbb{R}$. We rewrite the fractional Laplacian as
\begin{align*}
[-(- \Delta)^{\gamma/2} G ] (u ) =&   c_{\gamma}   \int_{0}^{\infty}  \frac{G(u+w)  + G(u-w) -2 G(u)}{ w^{\gamma+1}} dw.
\end{align*}
There are three possibilities to analyse: $u > 2b_G$, $u < - 2b_G$ and $- 2b \leq u \leq 2b_G$. First assume $u > 2b_G$. Since $G(y)=0$ when $u \geq b_G$, we get
\begin{align*}
 \Big| [-(- \Delta)^{\gamma/2} G ](u)\Big |=  c_{\gamma}  \Big| \int_{0}^{\infty}  \frac{G(u-w)}{ w^{\gamma+1}} dw \Big| = c_{\gamma}  \Big|  \int_{ u-b_G}^{u+b_G}  \frac{ G(u-w) }{ w^{\gamma+1}} dw \Big| \leq H^G(u).
\end{align*}
The case $u < - 2 b_G$ is similar.  It remains to deal with the case $|u| \leq 2 b_G$. In this case, we get $0 = G(u+ w)=G(u-3w)$, for every $w >3b_G$. This leads to
\begin{align*}
[-(- \Delta)^{\gamma/2} G ](u)=&  c_{\gamma}   \int_{ 0}^{3b_G}  \frac{[G(u+w) - G(u) ]  + [ G(u-w) - G(u) ]}{ w^{\gamma+1}} dw.
\end{align*}
Since $G \in C^2(\mathbb{R})$, performing two Taylor expansions of second order on $G$, we get 
\begin{align*}
 \Big| [-(- \Delta)^{\gamma/2} G ] (u )\Big |    \leq&   c_{\gamma}   \int_{0}^{3b_G}  \frac{ \| \Delta G \|_{\infty}   }{ w^{\gamma-1}} dw = H^G(u),
\end{align*}
ending the proof.
\end{proof}
Finally, we state a result which is a consequence of the two previous ones.
\begin{cor} \label{corfrac}
Let $G \in \mcb S$. Then
\begin{align*}
 \lim_{n \rightarrow \infty} \frac{1}{n} \sum_{x} \sup_{s \in [0,T]}|  [ -(-\Delta)^{\gamma/2}G_s](\tfrac{x-1}{n}) - [ -(-\Delta)^{\gamma/2}G_s](\tfrac{x}{n}) |=0.
\end{align*}
\end{cor}
\begin{proof}
From  Proposition \ref{propL1}, it holds
\begin{equation} \label{limsup1}
 \limsup_{M \rightarrow \infty} \limsup_{n \rightarrow \infty} \frac{1}{n} \sum_{|x| \geq M n } \sup_{s \in [0,T]}|  [ -(-\Delta)^{\gamma/2}G_s](\tfrac{x-1}{n}) - [ -(-\Delta)^{\gamma/2}G_s](\tfrac{x}{n}) |=0.
\end{equation}
From Propositions \ref{propdif} and \ref{propL1}, there exists $K >0$ such that
\begin{align*}
\sup_{(s,u) \in [0,T] \times \mathbb{R}} \big| \big( [ -(-\Delta)^{\gamma/2}G_s] \big)^{(1)} (u) \big| \leq K.
\end{align*}
Then a Taylor expansion of first order leads to 
 \begin{equation} \label{limsup2}
 \limsup_{M \rightarrow \infty} \limsup_{n \rightarrow \infty} \frac{1}{n} \sum_{|x| < M n } \sup_{s \in [0,T]}|  [ -(-\Delta)^{\gamma/2}G_s](\tfrac{x-1}{n}) - [ -(-\Delta)^{\gamma/2}G_s](\tfrac{x}{n}) | \leq \limsup_{M \rightarrow \infty} \limsup_{n \rightarrow \infty} \frac{2MK}{n}=0.
\end{equation}
The proof ends combining \eqref{limsup1} and \eqref{limsup2}.
\end{proof}
\section{Uniqueness of weak solutions} \label{secuniq}
In this section we prove the uniqueness of the weak solutions of \eqref{fracpor} for $m\in\mathbb Z$ and $m\geq2$. 
 First we observe that from Theorem 7.38 in \cite{sobolevadams}, we know  that $C_c^{\infty}(\mathbb{R})$ is dense in $\mcb{H}^{\gamma/2}$ with respect to the norm $\| \cdot \|_{\mcb{H}^{\gamma/2}}$. Moreover, from Proposition 23.2 (d) in \cite{MR1033497}, we also know  that $P ([0,T], \mcb{H}^{\gamma/2})$ is dense in $L^2(0,T; \mcb{H}^{\gamma/2})$.
As a  corollary of these results we obtain that 
$\mcb S$ is dense in $L^2(0,T; \mcb{H}^{\gamma/2})$. We recall that weak solutions of \eqref{fracpor} deal with $\mcb S$ as the space of test functions and the uniqueness of the weak solutions of \eqref{fracpor} is equivalent to the following result. Recall the definition of  $F(t, \rho, G,  g )$ given in Definition \ref{eq:dif}.
\begin{prop} \label{uniqeqhydsdifrobgen}
Let $\rho_1, \rho_2$ be such that $\rho_1 -b, \rho_2 -b  \in L^2 \big( 0,T; L^2 (\mathbb{R}) \big)$ and $\rho_1^m -b^m, \rho_2^m -b^m  \in L^2 ( 0,T; \mcb{H}^{\gamma/2} )$, for some $b \in (0,1)$. If 
$
F(t, \rho_1, G,   g  ) = 0 = F(t, \rho_2, G,  g )$, for every $t \in [0,T]$ and every $G \in \mcb S$, then $\rho_1 = \rho_2$ almost everywhere in $[0,T] \times \mathbb{R}$.
\end{prop}
\begin{proof}
Denote $\rho_3:= \rho_1 - \rho_2= [\rho_1 - b] - [\rho_2 - b]$. Then  $\rho_3 \in  L^2 \big( 0,T; L^2 (\mathbb{R}) \big)$. Moreover, denote $\rho_4:= \sum_{k=0}^{m-1} \rho_1^{k} \rho_2^{m-1-k}$ and $\rho_5:= \rho_3 \cdot \rho_4$. Then $\rho_5 =  [\rho_1^m - b^m] - [\rho_2^m - b^m]$, $\rho_5 \in  L^2 ( 0,T; \mcb{H}^{\gamma/2} )$ and $\rho_5(s, \cdot) \in \mcb{H}^{\gamma/2}$, for almost every  $ s \in [0,T]$. For every $t \in [0,T]$, for every $G \in \mcb S$ we get that $0=F(t, \rho_1, G,   g  )- F(t, \rho_2, G,  g ) $ is equivalent to
\begin{equation} \label{uniqgen}
0=  \int_{\mathbb{R}} \rho_3(t,u) G(t,u) du -  \int_0^t \int_{\mathbb{R}} \rho_{3}(s,u)  \partial_s  G(s,u) du ds -   \int_0^t \int_{\mathbb{R}} \rho_{5}(s,u) [ - (- \Delta)^{\gamma/2} G](s,u) du ds.
\end{equation}
Since $\rho_{5}  \in  L^2 ( 0,T; \mcb{H}^{\gamma/2} )$, there exists $(H_{k})_{k \geq 1}$  in $\mcb {S}$ such that $(H_{k})_{k \geq 1}$ converges to $\rho_{5}$ with respect to the norm of $L^2 ( 0,T; \mcb{H}^{\gamma/2} )$. Define $G_{k} (t,u) :=  \int_t^T H_{k} (s,u) ds$, for every $(t,u) \in [0,T] \times \mathbb{R}$ and for every $k \geq 1$.
In particular, $G_{k}  (T,u) = 0$, for every $u \in \mathbb{R}$ and every $k \geq 1$. Taking $t=T$  and $G=G_{k}$ in \eqref{uniqgen}, we get
	\begin{align}
	\forall k \geq 1, \quad	0=    - \int_0^T \int_{\mathbb{R}} \rho_{3}(s,u)  \partial_s  G_{k}(s,u) du ds
		- &  \int_0^t \int_{\mathbb{R}} \rho_{5}(s,u) [ - (- \Delta)^{\gamma/2} G_{k}](s,u) du ds. \label{equniqbarforgen}
	\end{align}
To  treat the rightmost term in last display,  we can use  Lemma B.9 of \cite{CGJ2}, that says that 
\begin{equation} \label{lemuniq2gen}
\lim_{k \rightarrow \infty}  \int_0^T \int_{\mathbb{R}} \rho_5(s,u) [ -(- \Delta)^{\gamma/2} G_k ] (s,u) du ds = -\frac{c_{\gamma}}{4}  \iint_{\mathbb{R}^2}  \frac{[  \int_0^T  \rho_5(r,u) dr - \int_0^T  \rho_5(s,v) ds ]^2  }{|u-v|^{1 + \gamma}} du dv .
\end{equation}
In order to treat the leftmost term we claim that
\begin{align} \label{limuniq1gen}
\lim_{k \rightarrow \infty} \int_0^T \int_{\mathbb{R}} \rho_{3}(s,u)  \partial_s  G_{k}(s,u) du ds = - \int_0^T \int_{\mathbb{R}} [\rho_{3}(s,u)]^2  \rho_{4}(s,u) du ds.
\end{align}
Indeed, from the definition of $(G_k)_{k \geq 1}$, we get $\partial_{s} G_k (s,u) = - H_k(s,u)$ for all $ s \in [0,T]$, for all $ u \in \mathbb{R}$ and for all $ k \geq 1$. This  leads to
\begin{align*}
&  \lim_{k \rightarrow \infty} \Big| \int_0^T \int_{\mathbb{R}} \rho_{3}(s,u)  \partial_s  G_{k}(s,u) du ds + \int_0^T \int_{\mathbb{R}}  [\rho_{3}(s,u)]^2  \rho_{4}(s,u) duds \Big|  \\
=& \lim_{k \rightarrow \infty}   \Big| \int_0^T \int_{\mathbb{R}} \rho_{3}(s,u) [ \rho_{5}(s,u)  - H_k(s,u)] du ds \Big|
\\
\leq& \lim_{k \rightarrow \infty} \sqrt{\int_0^T \int_{\mathbb{R}} [\rho_{3}(s,u)]^2 du ds} \sqrt{\int_0^T \int_{\mathbb{R}} [\rho_{5}(s,u)  - H_k(s,u)]^2 du ds}=0
\end{align*}
Above we used the Cauchy-Schwarz inequality, the fact that  $\rho_3 \in  L^2 \big( 0,T; L^2 (\mathbb{R}) \big)$ and that $(H_k)_{k \geq 1}$ converges to $\rho_5$ in $L^2 ( 0,T; \mcb{H}^{\gamma/2})$.
From last results, taking $k \rightarrow \infty$ in \eqref{equniqbarforgen}, we obtain 
\begin{align*}
&  \int_0^T \int_{\mathbb{R}} [\rho_{3}(s,u)]^2  \rho_{4}(s,u) du ds + \frac{c_{\gamma}}{4}  \iint_{\mathbb{R}^2}  \frac{[  \int_0^T  \rho_{5}(r,u) dr - \int_0^T  \rho_{5}(s,v) ds ]^2  }{|u-v|^{1 + \gamma}} du dv =0.
\end{align*}
Since $[\rho_{3}]^2  \rho_{4}  \geq 0$ on $[0,T] \times \mathbb{R}$, we have  $[\rho_{3}]^2=[\rho_{1}-\rho_{2}]^2=0$ or $ \rho_4=  \sum_{k=0}^{m-1} \rho_1^{k} \rho_2^{m-1-k}=0$, and both  imply   $\rho_1 = \rho_2$ almost everywhere in $[0,T] \times \mathbb{R}$. 
\end{proof}

\thanks{ {\bf{Acknowledgements: }}
P.C. and R.P. thank FCT/Portugal for financial support through the project Lisbon Mathematics Ph.D. (LisMath). P.C. and P.G. thank  FCT/Portugal for financial support through CAMGSD, IST-ID,
projects UIDB/04459/2020 and UIDP/04459/2020.  This project has received funding from the European Research Council (ERC) under  the European Union's Horizon 2020 research and innovative programme (grant agreement   n. 715734).} No new data were created or analysed in this study.

\bibliographystyle{plain}
\bibliography{DCG_final}

\end{document}